\documentclass[11pt]{amsart}
\usepackage{geometry}                
\usepackage{amsmath}
\usepackage{amscd}
\usepackage{amsthm}
\usepackage{graphicx}
\usepackage{amssymb}
\usepackage{epstopdf}
\theoremstyle{plain}
\newtheorem{Thm}{Theorem}[section]

\newtheorem{Prop}[Thm]{Proposition}
\newtheorem{Cor}[Thm]{Corollary}
\newtheorem{Lem}[Thm]{Lemma}

\theoremstyle{definition}
\newtheorem{Defn}[Thm]{Definition}
\newtheorem{Expl}[Thm]{Example}

\theoremstyle{remark}
\newtheorem{Rem}[Thm]{Remark}

\numberwithin{equation}{section}

\title{On multi-pointed non-commutative deformations and Calabi-Yau threefolds}
\author{Yujiro Kawamata}
\begin{document}
\maketitle


\begin{abstract}
We will develop a theory of multi-pointed non-commutative deformations of a simple collection 
in an abelian category, and construct relative exceptional objects and relative spherical objects in some cases.
This is inspired by a work by Donovan and Wemyss.
\end{abstract}

\section{introduction}

We shall develop a theory of multi-pointed non-commutative deformations of a simple collection 
in an abelian category.
A simple collection is a finite set of objects such that each object has no endomorphisms except dilations and 
there are no nonzero homomorphisms between objects.
The commutative deformations of several objects are just the direct product of deformations of each objects, 
but there are interactions of objects in the case of non-commutative deformations.
We will prove that any iterated non-trivial extensions between the given objects 
yield a non-commutative deformation in the case of a simple collection, 
and we obtain a versal deformation in this way.
As applications, we will construct relative exceptional objects and relative spherical objects in some special cases.

The deformation theory has non-commutative versions in two directions, non-commutative fibers 
and non-commutative base.
We consider the latter case.
The point is that there are more non-commutative deformations of commutative objects than the commutative deformations
as proved in a paper by Donovan and Wemyss \cite{Donovan-Wemyss1}. 
They discovered an interesting application of 
the theory of non-commutative deformations to the theory of three dimensional algebraic varieties.
They provided a better understanding of the 
mysterious analytic neighborhood of a flopping curve on 
a threefold by investigating non-commutative deformations of the flopping curve.
The invariants defined by them are found to be related to Gopakumar-Vafa invariants and
Donaldson-Thomas invariants (\cite{Toda-GV}).
This paper is motivated by their works.
Moreover we consider systematically multi-pointed deformations, i.e., non-commutative deformations of 
several objects.

The theory of deformations over a non-commutative base is developed by Laudal \cite{Laudal}.
The definition of non-commutative deformations is very similar to the commutative deformations, but 
only the parameter algebra is not necessarily commutative.
A non-commutative Artin semi-local algebra with nilpotent Jacobson radical is not necessarily 
a direct product of Artin local algebras.
By this reason, we need to consider several maximal ideals simultaneously.

The extensions of a deformation and the obstruction theory is similarly described by 
cohomology groups as in \cite{Schlessinger}, 
and there exists a versal family of non-commutative deformations under some mild conditions. 
But there are much more non-commutative deformations than the commutative ones.
For example, unobstructed deformations in 
the commutative case can be obstructed in the non-commutative sense.

Let $\Bbbk$ be a field, $\mathcal{A}$ a $\Bbbk$-linear abelian category, $r$ a positive integer, 
and $F_i$ ($1 \le i \le r$) objects in $\mathcal{A}$.
The set $\{F_i\}$ is said to be a {\em simple collection} if $\dim \text{Hom}(F_i,F_j) = \delta_{ij}$. 
We define non-commutative deformations of the collection $\{F_i\}$ 
as iterated non-trivial mutual extensions of the $F_i$.
We will prove that the non-commutative deformations behave very nicely under the condition of simplicity. 

In \S 2, we define a multi-pointed non-commutative deformation of a collection of objects.
In \S 3, we treat non-commutative deformations of objects as their iterated extensions.
First theorem states that, for any two sequences of iterated non-trivial extensions of a simple collection, 
there exists a third sequence of iterated non-trivial extensions
which dominates others (Theorem~\ref{dominates}).
In particular, if the extensions terminate, then there exists a unique versal deformation.

In the second theorem in \S 4, we prove the converse statement that arbitrary sequence of 
iterated non-trivial extensions of a simple collection 
can be regarded as a non-commutative deformation.
The point is that the base ring of the deformation is recovered as the ring of endomorphisms.
For this purpose, we consider a tower of universal extensions of a simple collection,
and we prove the flatness of the extension over the ring of endomorphisms.
In this way we construct a versal multi-pointed non-commutative deformation (Theorem~\ref{flat}).

As applications we construct relative multi-pointed exceptional objects and relative multi-pointed spherical objects 
in some special cases in \S 5 and \S 6.
A relative multi-pointed exceptional object yields a semi-orthogonal decomposition of a triangulated category, and 
a relative multi-pointed spherical object yields a twist functor.
In the case of a local Calabi-Yau threefold, we will prove that a versal non-commutative deformation of a
simple collection becomes a relative spherical object if the deformations stop after a finite number of steps.

We will use the abbreviation \lq\lq NC'' for non-commutative, or more precisely,  
not necessarily commutative in the rest of the paper.

The author would like to thank Alexei Bondal, Yukinobu Toda and Michael Wemyss 
for useful discussions on the subject.
Yukinobu Toda made a remark on the universality of the deformations, and 
Michael Wemyss pointed out that one needs to assume the rigidity of the curves in Example 6.7.
This work was partly done while the author stayed at National Taiwan University.
The author would like to thank Professor Jungkai Chen and 
National Center for Theoretical Sciences of Taiwan of the hospitality and 
excellent working condition.
I would like to thank the anonymous referees for careful reading of the first draft of this paper 
and many suggestions for improvements.

This work is partly supported by Grant-in-Aid for Scientific Research (A) 16H02141.

\section{definition of $r$-pointed NC deformations}

We give a definition of multi-pointed NC deformations.
It is modified from \cite{Laudal} in order to adapt to our situation of deformations of sheaves.
It seems that our treatment is also different from \cite{Donovan-Wemyss2},
because our definition works well only in the case of simple collections.
See also \cite{BB}.

We would like to consider infinitesimal deformations of $r$ coherent sheaves on a variety 
at the same time for a positive integer $r$.
If we consider only commutative deformations of these sheaves, 
then they deform independently.
But NC deformations reflect interactions among the sheaves.

First we define the category of base rings for deformations according to \cite{Laudal}:

\begin{Defn}
Let $\Bbbk$ be a base field, let $r$ be a positive integer, and let $\Bbbk^r$ be the direct product ring.
An {\em $r$-pointed $\Bbbk$-algebra} $R$ is an associative ring endowed 
with $\Bbbk$-algebra homomorphisms
\[
\Bbbk^r \to R \to \Bbbk^r
\]
whose composition is the identity homomorphism. 
\end{Defn}

Let $e_i$ be the idempotents of $R$ corresponding to the vectors 
$(0, \dots,0,1,0,\dots,0) \in \Bbbk^r$ for $1 \le i \le r$, 
where $1$ is at the $i$-th place.
We have $\sum_{i=1}^r e_i = 1$, $e_ie_i=e_i$ and $e_ie_j=0$ for $i \ne j$.
Let $R_{ij} = e_iRe_j \subset R$.
Then $R = \bigoplus_{i,j=1}^r R_{ij}$, and $R$ can be considered as a matrix algebra $(R_{ij})$ such that
the $R_{ij}$ are $\Bbbk$-vector spaces and the multiplication in $R$ is given by 
$\Bbbk$-linear homomorphisms 
$R_{ij} \otimes_{\Bbbk} R_{jk} \to R_{ik}$.

Let $M_i$ be the kernels of the surjective algebra homomorphisms 
$R \to \Bbbk^r \to \Bbbk$ for $1 \le i \le r$, where the 
second homomorphisms are $i$-th projections.
These are maximal ideals and the $R/M_i$ are simple two-sided $R$-modules.
Let $M = \bigcap M_i$. 
We have $M = \text{Ker}(R \to \Bbbk^r)$ and $R/M = \bigoplus R/M_i$ as $R$-modules.

\begin{Defn}
We define $(\text{Art}_r)$ to be the category of $r$-pointed $\Bbbk$-algebras $R$ 
such that $\dim_{\Bbbk} R < \infty$ and 
$M$ is nilpotent.
\end{Defn}

The second condition is independent.
For example, let $R = \Bbbk \oplus \Bbbk$ be a commutative $\Bbbk$-algebra 
with $r=1$ and $M = 0 \oplus \Bbbk$.
Then $M$ is not nilpotent.

If $R \in (\text{Art}_r)$, then any simple right $R$-module is isomorphic to some $R/M_i$.
Indeed, let $N = R/I$ be a simple module for a right ideal $I$.
Since $M$ is nilpotent, there is an integer $i$ such that $M^i \not\subset I$ but $M^{i+1} \subset I$.
Then there is an element $n \in N$ such that $nM = 0$ in $N$.
Then $\text{Ann}(n)/M$ is a right ideal of $R/M \cong \Bbbk^r$, that is one of the $M_i/M$.  

\begin{Defn}
Let $\mathcal{A}$ be a $\Bbbk$-linear abelian category.
An object $F$ of $\mathcal{A}$ has a left $R$-module structure if 
there is a $\Bbbk$-linear map $R \to \text{End}(F)$.
For a right $R$-module $N$ with presentation $R^{(I)} \to R^{(J)} \to N \to 0$, we define a tensor product 
$N \otimes_R F$ 
as the cokernel of $F^{(I)} \to F^{(J)}$.
$F$ is said to be {\em flat} if the exactness of a sequence $0 \to N_1 \to N_2 \to N_3 \to 0$ 
of right $R$-modules implies
the exactness of a sequence $0 \to N_1 \otimes_R F \to N_2 \otimes_R F \to N_3 \otimes_R F \to 0$. 

A set of objects $\{F_i\}_{i=1}^r$ in $\mathcal{A}$ is said to be a {\em collection} in this paper.
Let $F = \bigoplus F_i$. 
An {\em $r$-pointed NC deformation} of the collection $\{F_i\}$ over 
$R \in (\text{Art}_r)$ is a pair $(F_R,\phi)$ consisting of 
an object $F_R$ of $\mathcal{A}$ which has a flat left $R$-module structure
and an isomorphism $\phi: R/M \otimes_R F_R \cong F$ inducing isomorphisms 
$R/M_i \otimes_R F_R \cong F_i$ for all $i$.
The {\em $r$-pointed NC deformation functor} $\text{Def}_{\{F_i\}}: (\text{Art}_r) \to (Set)$ of $\{F_i\}$ 
is defined to be a covariant functor which sends $R$ to the set of isomorphism classes of $r$-pointed 
NC deformations of $\{F_i\}$ over $R$.
If $R \to R'$ is a $\Bbbk$-algebra homomorphism, then $r$-pointed deformations $F_R$ over $R$ 
are mapped to $r$-pointed deformations $R' \otimes_R F_R$ over $R'$.
\end{Defn}

For example, we can take $\mathcal{A} = (\text{coh}(X))$, 
the category of coherent sheaves on an algebraic variety $X$
defined over $\Bbbk$.

We give a definition of a versal deformation:

\begin{Defn}
A sequence of $r$-pointed NC deformations $\{(F^{(n)}, \phi^{(n)})\}$ of $\{F_i\}$ 
over rings $(R^{(n)},M^{(n)}) \in (\text{Art}_r)$
with surjective homomorphisms of $r$-pointed algebras $f_{n,n+1}: R^{(n+1)} \to R^{(n)}$ such that 
$(F^{(n)},\phi^{(n)}) \cong R^{(n)} \otimes_{R^{(n+1)}} (F^{(n+1)},\phi^{(n+1)})$ 
is called a {\em versal deformation} if the following conditions are satisfied:

\begin{enumerate}

\item For any $r$-pointed deformation $(F_{R'},\phi')$ of $\{F_i\}$ over a ring $R' \in (\text{Art}_r)$, 
there exist a positive integer $n$ and a ring homomorphism of $r$-pointed algebras 
$g: R^{(n)} \to R'$ such that 
$(F_{R'},\phi') \cong R' \otimes_{R^{(n)}} (F^{(n)},\phi^{(n)})$. 

\item There exists a positive integer $n$ such that the natural homomorphisms 
$M^{(n')}/(M^{(n')})^2 \to M^{(n)}/(M^{(n)})^2$ are bijective for all $n' \ge n$, and the induced
homomorphism $dg: M^{(n)}/(M^{(n)})^2 \to M_{R'}/M_{R'}^2$ is uniquely determined,
i.e., $dg = dg'$ for any other choice $g'$ satisfying (1).

\end{enumerate}
\end{Defn}

It follows from the condition (2) that, for each $m$, there exists $n(m)$ 
such that the natural homomorphisms
$(M^{(n')})^m/(M^{(n')})^{m+1} \to (M^{(n(m))})^m/(M^{(n(m))})^{m+1}$ are bijective for all $n' \ge n(m)$.
Indeed, there are surjective homomorphisms
$(M^{(n)}/(M^{(n)})^2)^{\otimes m} \to (M^{(n')})^m/(M^{(n')})^{m+1} \to (M^{(n)})^m/(M^{(n)})^{m+1}$
for any $n' \ge n$, so that $\dim (M^{(n')})^m/(M^{(n')})^{m+1}$ stabilize for large $n'$.

We have uniqueness of a versal deformation:

\begin{Prop}
Let $\{(F^{(n)}_j, \phi^{(n)}_j)\}$ ($j=1,2$) be
versal $r$-pointed NC deformations of $F$ over rings $(R^{(n)}_j,M^{(n)}_j) \in (\text{Art}_r)$, and
let $\hat{R}_j = \varprojlim R^{(n)}_j$ be the inverse limits.
Then there is an isomorphism $f: \hat{R}_1 \to \hat{R}_2$ which induces an
isomorphism $\varprojlim F^{(n)}_1 \cong \hat{R}_1 \hat{\otimes}_{\hat{R}_2} \varprojlim F^{(n)}_2$. 
\end{Prop}

\begin{proof}
By the versality, we have ring homomorphisms $f: \hat{R}_1 \to \hat{R}_2$ and $g: \hat{R}_2 \to \hat{R}_1$
which induces isomorphisms $\varprojlim F^{(n)}_1 \cong \hat{R}_1 \hat{\otimes}_{\hat{R}_2} \varprojlim F^{(n)}_2$
and $\varprojlim F^{(n)}_2 \cong \hat{R}_2 \hat{\otimes}_{\hat{R}_1} \varprojlim F^{(n)}_1$.
It is enough to prove that $g \circ f$ and $f \circ g$ are bijective.
We know that $g \circ f$ induces the identity on $M_{\hat R_1}/M_{\hat R_1}^2$.
Therefore it follows that it induces surjections on $M_{\hat R_1}^m/M_{\hat R_1}^{m+1}$ for all $m$.
Since $M_{\hat R_1}/M_{\hat R_1}^2$ is finite dimensional as a $\Bbbk$-module, 
so are the $M_{\hat R_1}^m/M_{\hat R_1}^{m+1}$.
Therefore they are also injective.
Hence $g \circ f$ is bijective.
$f \circ g$ is bijective similarly.
\end{proof}

\begin{Rem}
(1) There is a hull $\hat R$ for the functor $\text{Def}_{\{F_i\}}$ under suitable conditions (\cite{Laudal}).
If $r=1$, then the maximal commutative quotient $(\hat R)_{\text{ab}}$ coincides with the hull of the 
usual commutative deformation functor.
$\hat R$ is determined by $\text{Ext}^1(F,F)$ and the Massey products
$(\text{Ext}^1(F,F))^{\otimes m} \to \text{Ext}^2(F,F)$ for $m \ge 2$ (\cite{Laudal2}).
We will not use these facts.

(2) NC deformations exist only over local base by definition.
But Kapranov and Toda constructed globalization of NC deformations in the commutative direction 
(\cite{Toda-SLC}).
\end{Rem}

\section{iterated non-trivial $r$-pointed extensions}

We shall define the notion of a simple collection and consider its iterated non-trivial multi-pointed extensions.
A simple collection behaves well under iterated multi-pointed extensions. 

\begin{Defn}
Let $\mathcal{A}$ be a $\Bbbk$-linear abelian category, and let $r$ be a positive integer.
A collection $\{F_i\}_{i=1}^r$ in $\mathcal{A}$ is said to be a {\em simple collection} 
if $\dim \text{Hom}(F_i,F_j) = \delta_{ij}$.
\end{Defn}

If $\mathcal{A}$ is a category of coherent sheaves on a variety, 
then a member of a simple collection is usually called a {\em simple sheaf}.
This is the origin of the term \lq\lq simple''.
But we note that a simple sheaf is not necessary a simple object in the abelian category of sheaves. 

We consider iterated non-trivial $r$-pointed extensions of a simple collection $\{F_i\}_{i=1}^r$.

\begin{Defn}
A sequence of {\em iterated non-trivial $r$-pointed extensions} of the simple collection $\{F_i\}_{i=1}^r$ 
is a sequence of objects
$\{G^n\}_{0 \le n \le N}$ for a positive integer $N$ with decompositions $G^n = \bigoplus_{i=1}^r G^n_i$ such that
$G^0_i = F_i$, and for each $0 \le n < N$, there are $i = i(n)$ and $j = j(n)$ such that  
\[
0 \to F_j \to G^{n+1}_i \to G^n_i \to 0
\]
is an extension corresponding to a non-zero element of $\text{Ext}^1(G^n_i,F_j)$, and 
$G^{n+1}_{i'} = G^n_{i'}$ for $i' \ne i$.
\end{Defn}

We prove that any two iterated non-trivial $r$-pointed extensions are dominated by a third:

\begin{Thm}\label{dominates}
Let $\{F_i\}$ be a simple collection, and let $G$ be an object.
Let $0 \to F_{i_j} \to G_j \to G \to 0$ for $j=0,1$ be two non-trivial extensions
which are not isomorphic.
Then there exists a common object $H$ with non-trivial extensions  
$0 \to F_{i_{1-j}} \to H \to G_j \to 0$.
\end{Thm}

\begin{proof}
Let $\xi_j \in \text{Ext}^1(G,F_{i_j})$ 
be non-zero elements corresponding to the given extensions.
We consider exact sequences
\[
\text{Hom}(F_{i_{1-j}},F_{i_j}) \to \text{Ext}^1(G,F_{i_j}) \to \text{Ext}^1(G_{1-j},F_{i_j}) 
\]
derived from $\xi_{1-j}$.
Let $\xi'_j \in \text{Ext}^1(G_{1-j},F_{i_j})$ be the images of $\xi_j$ by the second homomorphism.
We claim that $\xi'_j \ne 0$.  
Indeed if $i_0 \ne i_1$, then the first term vanishes, hence $\xi'_j \ne 0$.
If $i_0 = i_1$, then the image of the first homomorphism is generated by $\xi_{1-j}$, hence
the image of $\xi_j$ by the second homomorphism is non-zero because the two extensions are not isomorphic.

We have a commutative diagram
\[
\begin{CD}
@. @. 0 @. 0 @. \\
@. @. @VVV @VVV @. \\
@. @. F_{i_1} @>=>> F_{i_1} @. \\
@. @. @VVV @VVV @. \\
0 @>>> F_{i_0} @>>> H @>>> G_1 @>>> 0 \\
@. @V=VV @VVV @VVV @. \\
0 @>>> F_{i_0} @>>> G_0 @>>> G @>>> 0 \\
@. @. @VVV @VVV @. \\
@. @. 0 @. 0 @.
\end{CD}
\]
where the two horizontal short exact sequences correspond to $\xi'_0$ and $\xi_0$.
They are commutative by the construction of $\xi'_0$.
By 9-lemma, we obtain the two vertical short exact sequences, which correspond to 
$\xi'_1$ and $\xi_1$.
Therefore we have constructed a common non-trivial extension $H$.

$H$ can be directly constructed as the kernel of the morphism $p_0-p_1: G_0 \oplus G_1 \to G$ 
where the $p_i: G_0 \oplus G_1 \to G_i \to G$ are the given morphisms.
\end{proof}

The maximal iterated non-trivial $r$-pointed extension is unique if it exists:

\begin{Cor}
Let $\{G^m\}_{0 \le m \le M}$ and $\{H^n\}_{0 \le n \le N}$ 
be two sequences of iterated non-trivial $r$-pointed extensions 
of a simple collection $\{F_i\}$.
Assume that $\text{Ext}^1(G^M,F_i) = 0$ for all $i$.
Then $M \ge N$, and there exists a sequence of iterated non-trivial $r$-pointed extensions 
$\{H^n\}_{0 \le n \le M}$ extending
the given sequence such that $G^M \cong H^M$.
In particular, the maximal iterated non-trivial $r$-pointed extension is unique if it exists.
\end{Cor}

\begin{proof}
Let $N' \le \min\{M,N\}$ be the maximum number 
such that the sequences $\{G^m\}_{0 \le m \le N'}$ and $\{H^n\}_{0 \le n \le N'}$ 
are isomorphic as iterated non-trivial $r$-pointed extensions.
We will prove that there exists another sequence $\{(G')^m\}_{0 \le m \le M}$ such that 
$(G')^M \cong G^M$ and $N'' > N'$ for the maximum number $N''$ where
the sequences $\{(G')^m\}_{0 \le m \le N''}$ and $\{H^n\}_{0 \le n \le N''}$ 
are isomorphic.
Then we obtain our assertion by the induction.

We will obtain the new sequence $\{(G')^m\}_{0 \le m \le M}$ by replacing the extensions inductively.
Namely, let $L^1$ be a common extension of $G^{N'+1}$ and $H^{N'+1}$ of $G^{N'} \cong H^{N'}$ 
given by the theorem.
If $L^1 \cong G^{N'+2}$, then we replace $G^{N'+1}$ by $H^{N'+1}$ and leave other $G^m$'s unchanged.
Otherwise we take a common extension $L^2$ of $G^{N'+2}$ and $L^1$ over $G^{N'+1}$.
If $L^2 \cong G^{N'+3}$, then we replace $G^{N'+1}$ and $G^{N'+2}$ by $H^{N'+1}$ and $L^1$, respectively, 
and leave other $G^m$'s unchanged.
Otherwise we take a common extension $L^3$ of $G^{N'+3}$ and $L^2$ over $G^{N'+2}$.
Since $\text{Ext}^1(G^M,F_i) = 0$ for all $i$, this process stops after finitely many repetition, hence our result.
\end{proof}

\begin{Rem}
(1) The above theorem is the reason why our theory works well only for simple collections.

(2) The sheaf $G^M$ in the above corollary will be proved to be a versal $r$-pointed NC deformation of the 
simple collection $\{F_i\}$ in the case where the base ring is finite dimensional in Theorem~\ref{flat}.
The base ring of the deformation will be constructed in the next section.
The general case where the sequence of iterated non-trivial $r$-pointed extensions 
may not terminate will also be treated there.
\end{Rem}

We will need the following in the next section:

\begin{Lem}\label{no more hom}
Let $\{G^n\}$ with $G^n = \bigoplus_i G^n_i$ 
be a sequence of iterated non-trivial $r$-pointed extensions of a simple collection $\{F_i\}$.
Then $\dim \text{Hom}(G^n_i,F_j) = \delta_{ij}$ for all $i,j,n$.
\end{Lem}

\begin{proof}
We proceed by induction on $n$.
If $n=0$, then the assertion is true by the assumption of the simplicity.
Suppose that we have an exact sequence
\[
0 \to F_j \to G^{n+1}_i \to G^n_i \to 0.
\]
Then we have a long exact sequence
\[
0 \to \text{Hom}(G^n_i,F_k) \to \text{Hom}(G^{n+1}_i,F_k) \to 
\text{Hom}(F_j,F_k) \to \text{Ext}^1(G^n_i,F_k)
\]
for any $k$.
If $k \ne j$, then the third term vanishes, hence the first arrow is bijective.
If $k = j$, then the last arrow is injective because the extension is non-trivial.
Therefore the first arrow is bijective again.
Hence we conclude the proof.
\end{proof}

\section{iterated universal $r$-pointed extensions}

We will construct a sequence of universal $r$-pointed extensions of a simple collection $\{F_i\}$ 
under the assumption that $\dim \text{Ext}^1(F,F) < \infty$, and 
prove the existence of a versal $r$-pointed NC deformation.

\begin{Prop}\label{universal}
Let $\{F_i\}_{r=1}^r$ be a simple collection, let $F= \bigoplus_{i=1}^r F_i$ be the sum of the collection, 
and set $F = F^{(0)}$ and $F_i = F_i^{(0)}$.
Assume that $\dim \text{Ext}^1(F,F) < \infty$.
Then there exists a sequence of universal extensions $F^{(n)}  =  \bigoplus_{i=1}^r F_i^{(n)}$ 
given by 
\[
0 \to \bigoplus_j \text{Ext}^1(F_i^{(n)},F_j)^* \otimes F_j \to F_i^{(n+1)} \to F_i^{(n)} \to 0
\]
for each $i$, which is also obtained by a sequence of iterated non-trivial $r$-pointed extensions of the collection $\{F_i\}$.
\end{Prop}

\begin{proof}
A natural morphism 
$F_i^{(n)} \to \bigoplus_j \text{Ext}^1(F_i^{(n)},F_j)^* \otimes F_j [1]$ in the derived category 
$D(\mathcal{A})$ of the abelian category $\mathcal{A}$ yields the extension as stated in the proposition.
We will prove that this extension is obtained as an output of a 
sequence of iterated non-trivial $r$-pointed extensions of the collection $\{F_i\}$.
We note that we have $\dim \text{Ext}^1(F_i^{(n)},F_j) < \infty$ for all $i,j,n$ under the assumption.

We write $G = F_i^{(n)}$.
We take a basis $\{v_{j,1},\dots,v_{j,M_j}\}$ of $\text{Ext}^1(G,F_j)$ for each $j$, 
and let $V_{j,m} \subset \text{Ext}^1(G,F_j)$ be
the subspaces generated by $v_{j,1},\dots,v_{j,m}$ for $1 \le m \le M_j$.
We set $V_{j,0}=0$.
The natural morphisms 
\[
G \to \bigoplus_{k=1}^{j-1} \text{Ext}^1(G,F_k)^* \otimes F_k[1] \oplus V_m^* \otimes F_j[1] 
\to \bigoplus_{k=1}^{j-1} \text{Ext}^1(G,F_k)^* \otimes F_k[1] \oplus V_{m-1}^* \otimes F_j[1]
\]
yield a commutative diagram of extensions 
\[
\begin{CD}
0 @>>> \bigoplus_{k=1}^{j-1} \text{Ext}^1(G,F_k)^* \otimes F_k \oplus V_m^* \otimes F_j 
@>>>  G^{M+m} @>>> G @>>> 0 \\
@. @VVV @VVV @V=VV @. \\
0 @>>> \bigoplus_{k=1}^{j-1} \text{Ext}^1(G,F_k)^* \otimes F_k \oplus V_{m-1}^* \otimes F_j 
@>>>  G^{M+m-1} @>>> G @>>> 0
\end{CD}
\]
where $M = \sum_{k=1}^{j-1}M_k$.
Thus we obtain extensions
\[
0 \to F_j \to G^{M+m} \to G^{M+m-1} \to 0.
\]

We will prove that these extensions are non-trivial.
Since $\text{Hom}(F_{j'},F_j) \cong \Bbbk^{\delta_{j'j}}$, we have a commutative diagram of exact sequences
\[
\begin{CD}
V_{m-1} @>>> \text{Ext}^1(G,F_j) @>>> \text{Ext}^1(G^{m-1},F_j) \\
@VVV @V=VV @VVV \\
V_m @>>> \text{Ext}^1(G,F_j) @>>> \text{Ext}^1(G^m,F_j)  
\end{CD}
\]
where the last vertical arrow has a non-trivial kernel generated by the image of $v_{j,m}$.
Therefore the extensions above are non-trivial.

A referee pointed out that 
the above universal extension can be expressed as 
\[
0 \to \text{Ext}^1(F^{(n)},F)^* \otimes_{R_0} F \to F^{(n+1)} \to F^{(n)} \to 0
\]
for $R_0 = \text{End}(F)=\Bbbk^r$.
\end{proof}

\begin{Defn}
We define a filtration of $F^{(n)}$ by $G^p(F^{(n)}) = \text{Ker}(F^{(n)} \to F^{(p-1)})$ for $0 \le p \le n+1$.
We have $G^0(F^{(n)}) = F^{(n)}$ and $G^{n+1}(F^{(n)}) = 0$.
\end{Defn}

Let $\text{End}_G(F^{(n)})$ be the ring of endomorphisms of the object $F^{(n)}$ 
which preserve the filtration $\{G^p\}$.

\begin{Lem}\label{End-surjective}
The natural ring homomorphism $\text{End}_G(F^{(n+1)}) \to \text{End}_G(F^{(n)})$ is surjective. 
\end{Lem}

\begin{proof}
Since we only consider endomorphisms preserving the filtration, there is certainly a natural ring homomorphism. 
For any element $f \in \text{End}_G(F^{(n)})$, 
we have a commutative diagram in the derived category $D(\mathcal{A})$:
\[
\begin{CD}
F^{(n+1)} @>>> F^{(n)} @>>> \bigoplus_j \text{Ext}^1(F^{(n)},F_j)^* \otimes F_j[1] \\
@. @VfVV @Vf^{**}VV \\
F^{(n+1)} @>>> F^{(n)} @>>> \bigoplus_j \text{Ext}^1(F^{(n)},F_j)^* \otimes F_j[1]. 
\end{CD}
\]
We obtain a lifting of $f$ to $\text{End}_G(F^{n+1})$ by the axiom of a triangulated category. 
\end{proof}

By Lemma~\ref{no more hom}, we have the following

\begin{Lem}
Let $r_0 = r$ and $r_{m+1} = \sum_j \dim \text{Ext}^1(F^{(m)},F_j)$ for $m \ge 0$.
Then $\dim \text{End}_G(F^{(n)}) = \sum_{m=0}^n r_m$.
\end{Lem}

\begin{proof}
We have $\dim \text{End}_G(F^{(0)}) = r$.
We have the following exact sequence
\begin{equation}\label{End}
0 \to \text{Hom}(F^{(m+1)},\bigoplus_j \text{Ext}^1(F^{(m)},F_j)^* \otimes F_j)
\to \text{End}_G(F^{(m+1)}) \to \text{End}_G(F^{(m)}). 
\end{equation}
The dimension of the first term is equal to 
$\sum_j \dim \text{Ext}^1(F^{(m)},F_j) = r_{m+1}$.
Therefore we conclude the proof. 
\end{proof}

\begin{Lem}
$\dim \text{End}(F^{(n)}) \le \sum_{m=0}^n r_m$.
\end{Lem}

\begin{proof}
By Lemma~\ref{no more hom}, we have $\dim \text{Hom}(F^{(n)}, F_i) = 1$ for all $i$.
Since the total number of the $F_i$'s in the extension process yielding $F^{(n)}$ is 
equal to $\sum_{m=0}^n r_m$, we deduce our inequality by exact sequences.
\end{proof}

\begin{Cor}\label{End-bijective}
The natural inclusion 
$\text{End}_G(F^{(n)}) \subseteq \text{End}(F^{(n)})$ is bijective.
\end{Cor}

Let $R^{(n)} = \text{End}(F^{(n)})$ and 
$R_{ij}^{(n)} = \text{Hom}(F_j^{(n)},F_i^{(n)})$.
Then we can write in a matrix form as $R^{(n)} = (R_{ij}^{(n)})$.
Let $M^{(n)} = \text{Ker}(R^{(n)} \to R^{(0)})$ and $M^{(n)}_i = \text{Ker}(R^{(n)} \to R^{(0)} \to \Bbbk)$, 
where the
last arrow is the projection to the $i$-th factor.

\begin{Prop}
$R^{(n)} \in (\text{Art}_r)$.
\end{Prop}

\begin{proof}
There is a ring homomorphism $R^{(n)} \to R^{(0)} \cong \Bbbk^r$.
The idempotent $e_i$ of $R^{(n)}$ coincides with the projection $F^{(n)} \to F^{(n)}_i \subset F^{(n)}$ to the 
$i$-th factor.
This gives the $\Bbbk^r$-algebra structure of $R^{(n)}$.
We know already that $\dim R^{(n)} < \infty$.

We will prove that $M^{(n)}$ is nilpotent by induction on $n$.
$M^{(0)} = 0$.
We assume that $(M^{(n)})^m = 0$ for some $m > 0$, and consider $M^{(n+1)}$.
By the assumption, $(M^{(n+1)})^m(F^{(n+1)}) \subset G^{n+1}(F^{(n+1)})$, 
where $G^{n+1}(F^{(n+1)}) = \bigoplus_j \text{Ext}^1(F^{(n)},F_j)^* \otimes F_j$.
There is an exact sequence
\[
0 \to \text{Hom}(F^{(0)},G^{n+1}(F^{(n+1)})) \to \text{Hom}(F^{(n+1)},G^{n+1}(F^{(n+1)}))
\to \text{Hom}(G^1(F^{(n+1)}),G^{n+1}(F^{(n+1)})).
\]
The first homomorphism is bijective by Lemma~\ref{no more hom}, hence the second homomorphism is zero. 
It follows that $(M^{(n+1)})^m(G^1(F^{(n+1)})) = 0$.
Therefore $(M^{(n+1)})^{2m} = 0$.
\end{proof}

\begin{Thm}\label{flat}
(1) The above constructed $F^{(n)}$ with a natural isomorphism 
$\phi^{(n)}: R^{(n)}/M^{(n)} \otimes_{R^{(n)}} F^{(n)} \cong F$
is an $r$-pointed NC deformation of the simple collection $\{F_i\}$ 
over the ring $R^{(n)}$. 

(2) The sequence $\{(F^{(n)}, \phi^{(n)})\}$ of $r$-pointed NC deformations over $\{(R^{(n)}, M^{(n)})\}$ is 
a versal deformation of the simple collection $\{F_i\}$.
\end{Thm}

\begin{proof}
(1) Since $\dim \text{Hom}(F^{(n)},F_j) = 1$, we have $\text{Hom}(F^{(n)},F_j) \cong R^{(n)}/M^{(n)}_j$ 
as right $R^{(n)}$-modules for all $i$. 
Indeed if $\text{Hom}(F^{(n)},F_j)$ is generated by the natural projection $f_j$, then
$M^{(n)}_j = \{s \in R^{(n)} \mid f_js = 0\}$.
Thus we have exact sequences of right $R^{(n)}$-modules
\[
0 \to \bigoplus_j \text{Ext}^1(F^{(k)},F_j)^* \otimes R^{(n)}/M^{(n)}_j \to R^{(k+1)} \to R^{(k)} \to 0
\]
for $0 \le k < n$ by (\ref{End}), where the last arrow is surjective by 
Lemma~\ref{End-surjective} and Corollary~\ref{End-bijective}.

We prove that $R^{(k)} \otimes_{R^{(n)}} F^{(n)} \cong F^{(k)}$ and
$\text{Tor}^1_{R^{(n)}}(R^{(k)}, F^{(n)}) = 0$ by the descending induction on $k$.
If $k = n$, then this is obvious.
By taking the tensor product $\otimes_{R^{(n)}} F^{(n)}$ with the above exact sequence, 
we obtain the universal extension exact sequence
\[
0 \to \bigoplus_j \text{Ext}^1(F^{(k)},F_j)^* \otimes F_j \to F^{(k+1)} \to F^{(k)} \to 0.
\]
Therefore, if our assertion is true for $k+1$ for some $0 \le k < n$, then it is also true for $k$.
If we set $k = 0$, then we obtain
\[
R^{(n)}/M^{(n)}_i \otimes_{R^{(n)}} F^{(n)} \cong F_i
\]
and 
\[
\text{Tor}^1_{R^{(n)}}(R^{(n)}/M^{(n)}_i, F^{(n)}) = 0
\]
for all $i$.
There are no simple $R^{(n)}$-modules other than the $R^{(n)}/M^{(n)}_i$, and 
any right $R^{(n)}$-module of finite type is an iterated extension of simple modules.
Therefore $F^{(n)}$ is flat over $R^{(n)}$.

(2) Let $F_R$ be an arbitrary $r$-pointed NC deformation of $F$ over $(R,M)$.
Since $M/M^2$, and hence $M^k/M^{k+1}$ for all $k$, are direct sums 
of simple modules $R/M_i \cong \Bbbk$,  
there exists a decreasing sequence of two sided ideals $\{I_m\}$ of $R$ 
such that $I_0 = M$, $I_n = 0$ and $I_m/I_{m+1} \cong \Bbbk$ for all $m$.
We will inductively construct $\Bbbk$-algebra homomorphisms $f_m: R^{(a_m)} \to R/I_m$ for some 
non-decreasing sequence of integers $a_m$ such that
we have isomorphisms $R/I_m \otimes_R F_R \cong R/I_m \otimes_{R^{(a_m)}} F^{(a_m)}$ 
over $R/I_m$ for all $m$.  

Assume that $f_m$ is already constructed, and let us extend it to $f_{m+1}$.
We consider an extension of algebras 
\[
0 \to I_m/I_{m+1} \to R/I_{m+1} \to R/I_m \to 0
\]
and the corresponding extension of objects
\[
0 \to F_i \to R/I_{m+1} \otimes_R F_R \to R/I_m \otimes_R F_R \to 0
\]
for some $i$.
Since there is an isomorphism $R/I_m \otimes_R F_R \cong R/I_m \otimes_{R^{(a_m)}} F^{(a_m)}$, 
the homomorphism $f_m$ induces a natural homomorphism 
$f_m^*: \text{Ext}^1(R/I_m \otimes_R F_R,F_i) \to \text{Ext}^1(F^{(a_m)},F_i)$.
Let 
\[
0 \to F_i \to G \to F^{(a_m)} \to 0
\]
be the induced extension given by 
$G = \text{Ker}((R/I_{m+1} \otimes_R F_R) \oplus F^{(a_m)} \to R/I_m \otimes_R F_R)$.
It is an NC deformation over $R' = \text{Ker}(R/I_{m+1} \oplus R^{(a_m)} \to R/I_m)$ 
(see Lemma~\ref{amalsum} below).
Let $p: R' \to R/I_{m+1}$ and $q: R' \to R^{(a_m)}$ be projections. 
Then we have $R/I_{m+1} \otimes_{R'} G \cong R/I_{m+1} \otimes_R F_R$ by $p$, and 
$R^{(a_m)} \otimes_{R'} G \cong F^{(a_m)}$ by $q$.
 
There are two cases.
If the extension is trivial as $G \cong F^{(a_m)} \oplus F_i$, then we have $R' \cong R^{(a_m)} \oplus R/M_i$.
There is a homomorphism $g: R^{(a_m)} \to R'$ given by the identity and the projection such that 
we have $R' \otimes_{R^{(a_m)}} F^{(a_m)} \cong G$ by $g$.
We set $a_{m+1}=a_m$.

If the extension is non-trivial, then there is a non-zero elemet $\xi \in \text{Ext}^1(F^{(a_m)},F_i)$ corresponding to $G$.
There is a homomorphism $g: R^{(a_m+1)} \to R'$ given by 
$\xi_*: \text{Ext}^1(F^{(a_m)},F_i)^* \to \Bbbk$ such that 
we have $R' \otimes_{R^{(a_m+1)}} F^{(a_m+1)} \cong G$ by $g$.
We set $a_{m+1}=a_m+1$.

In either case, if we set $f_{m+1} = p \circ g$, then we have 
\[
R/I_{m+1} \otimes_{R^{(a_{m+1})}} F^{(a_{m+1})} \cong R/I_{m+1} \otimes_{R'} G \cong R/I_{m+1} \otimes_R F_R.
\]
This is what to be proved.

For all $i$, we have induced extensions
\[
0 \to M/MM_i \otimes_R F_R \to R/MM_i \otimes_R F_R \to F \to 0
\]
where we have $M/MM_i \otimes_R F_R \cong F_i^{b_i}$ for some $b_i \ge 0$.
Since the extension 
\[
0 \to \bigoplus \text{Ext}^1(F,F_i)^* \otimes F_i \to F^{(1)} \to F \to 0
\]
is universal,
there are uniquely determined homomorphisms $\text{Ext}^1(F,F_i)^* \to M/MM_i$ over $\Bbbk$ such that
their sum induces the first extensions from the second.
Since $R^{(1)}/M^{(1)} \cong R/M \cong \Bbbk^r$ are generated by dilations, the sum of the homomorphisms are
uniquely extended to a $\Bbbk$-algebra homomorphism $R^{(1)} \to R/M^2$ as required. 
\end{proof}

\begin{Lem}\label{amalsum}
Let $F_{R_t}$ be $r$-pointed NC deformations of a collection $\{F_i\}$ 
over $R_t \in (\text{Art}_r)$ for $t = 0,1,2$.
Assume that there are ring homomorphisms $f_t: R_t \to R_0$ such that 
$R_0 \otimes_{R_t} F_{R_t} \cong F_{R_0}$ for $t = 1,2$.
Then $F_{R_3} = \text{Ker}(F_{R_1} \oplus F_{R_2} \to F_{R_0})$ 
is an $r$-pointed NC deformation of $\{F_i\}$ 
over $R_3 = \text{Ker}(R_1 \oplus R_2 \to R_0)$.  
\end{Lem}

\begin{proof}
We prove that $F_{R_3}$ is flat as a left $R_3$-module.
Let $M_{t,i}$ be the $i$-th maximal two-sided ideals of $R_t$, i.e., 
$R_t/M_{t,i} \otimes  F_{R_t} \cong F_i$ for $t = 0,1,2,3$.
We have $M_{3,i} = \text{Ker}(M_{1,i} \oplus M_{2,i} \to M_{0,i})$.
We will prove that $\text{Tor}_1^{R_3}(R_3/M_{3,i}, F_{R_3}) = 0$ for all $i$.

By using the induction on the length of $\text{Ker}(f_1)$, it is sufficient to treat the case where 
there is an exact sequence
\[
0 \to R_1/M_{1,j} \to R_1 \to R_0 \to 0
\]
for some $j$.
Then we have
\[
\begin{split}
&0 \to R_3/M_{3,j} \to R_3 \to R_2 \to 0 \\
&0 \to R_3/M_{3,j} \to M_{3,i} \to M_{2,i} \to 0 \\
&0 \to F_j \to F_{R_3} \to F_{R_2} \to 0.
\end{split}
\]
By the flatness of the $F_{R_t}$ over $R_t$ for $t=0,1,2$, we have exact sequences
\[
0 \to M_{t,i} \otimes_{R_t} F_{R_t} \to F_{R_t} \to F_i \to 0.
\]

We consider the following commutative diagram
\[
\begin{CD}
@. @. 0 \\
@. @. @VVV \\
@. F_j @>{\alpha}>> F_j @. @. \\
@. @V{\beta}VV @VVV \\
@. M_{3,i} \otimes_{R_3} F_{R_3} @>{\gamma}>> F_{R_3} @>>> F_i @>>> 0 \\
@. @VVV @VVV @VV=V \\
0 @>>> M_{2,i} \otimes_{R_2} F_{R_2} @>>> F_{R_2} @>>> F_i @>>> 0 \\
@. @VVV @VVV \\
@. 0 @. 0 
\end{CD}
\]
where the first and second columns and the second and third rows are exact.
Then the arrow $\alpha$ is surjective.
Since $F_j$ is simple, it is also injective.
Then $\beta$ is injective.
It follows that $\gamma$ is also injective.
Therefore $F_{R_3}$ is flat.
\end{proof}

\begin{Rem}
(1) The above argument gives an explicit construction of the pro-representable hull, i.e., 
the versal $r$-pointed NC deformations, 
for a simple collection in a $\Bbbk$-linear abelian category $\mathcal{A}$ as the inverse limit of the $F^{(n)}$.

(2) The presentation of the pro-representable hull by Massey products (\cite{Laudal2}) corresponds to 
the following exact sequences
\[
0 \to \text{Ext}^1(F^{(n+1)},F_k) \to 
\bigoplus_j \text{Ext}^1(F^{(n)},F_j) \otimes \text{Ext}^1(F_j,F_k) \to 
\text{Ext}^2(F^{(n)},F_k)
\]
where we obtain inductively injective homomorphisms 
$\text{Ext}^1(F^{(n)},F) \to (\text{Ext}^1(F,F))^{\otimes (n+1)}$.

(3) The above defined versal family is not universal due to the non-commutativity of the deformation rings.
The deformation functor is pro-representable if the following condition is satisfied:
for any surjective ring homomorphism $R \to R'$, the natural homomorphism 
$\text{Aut}_R(F_R) \to \text{Aut}_{R'}(F_{R'})$ for $F_{R'} = R' \otimes_R F_R$ is surjective.
Since $\text{End}_{\Bbbk}(F_R) \cong R$, 
it follows that $\text{Aut}_R(F_R)$ coincides with the group of units of the 
center $Z(R)$ of $R$.
Since $Z(R) \to Z(R')$ is not necessarily surjective, 
there is no universal NC deformation of a simple collection in general. 

Indeed, two different homomorphisms to the versal algebra may give rise to isomorphic deformations.
Let $F_R$ be an $r$-pointed NC deformation of $F$ over $(R,M)$.
Assume that there is a socle $I$ of $R$; $I$ is a two sided ideal such that $I \cong R/M_i$ for some $i$.
Let $\bar R = R/I$.
Let $f: R^{(n)} \to R$ be the homomorphism obtained in the above theorem, and let $\bar f: R^{(n)} \to \bar R$
be its restriction.
Assume that there is an invertible element $\alpha \in R$ which is not in the center of $R$ such that its image
$\bar{\alpha} \in \bar R$ is in the center.
Then $f' = \alpha f \alpha^{-1}: R^{(n)} \to R$ is different from $f$, induces $\bar f$, and induces $F_R$ as 
shown in the following commutative diagram: 
\[
\begin{CD}
0 @>>> F_i @>>> F_R @>>> F_{\bar R} @>>> 0 \\
@. @VV=V @VV{\alpha}V @VV{\bar{\alpha}}V @. \\
0 @>>> F_i @>>> F_R @>>> F_{\bar R} @>>> 0
\end{CD}
\]
where $\alpha$ induces a $\Bbbk$-isomorphism of $F_R$ which does not commute with the action of $R$.
\end{Rem}

We give some criteria for the versality:

\begin{Cor}
Let $F = \bigoplus F_i$ be a simple collection.

(1) Let $G$ be a final object in a sequence of iterated non-trivial $r$-pointed extensions of $F$
i.e. the object obtained as the final output of the sequence of extensions.
Assume that $\text{Ext}^1(G,F) = 0$.
Then $G$ is a versal $r$-pointed NC deformation of $F$. 

(2) Let $G$ be a final object in a sequence of iterated $r$-pointed extensions of $F$ 
which are not necessarily non-trivial.
Assume that $\text{Hom}(G,F_i) \cong \Bbbk$ and $\text{Ext}^1(G,F_i)=0$ for all $i$. 
Then $G$ is a versal $r$-pointed NC deformation of $F$.  
\end{Cor}

\begin{proof}
(1) The assertion follows from the existence and uniqueness of a versal $r$-pointed NC deformation
proved in Theorem~\ref{flat}.

(2) The condition $\text{Hom}(G,F_i) \cong \Bbbk$ for all $i$ implies that the extensions are non-trivial.
\end{proof}

\section{$r$-pointed relative exceptional objects}

Exceptional collections yield important examples of semi-orthogonal decompositions.
We extend the definition of an exceptional object to a relative version, 
and prove that it also yields a semi-orthogonal decomposition.

We note that, if $F_R$ is an $r$-pointed NC deformation of some collection over $R$, then 
$\text{Hom}(F_R,a)$ has a right $R$-module structure
for any $a \in \mathcal{A}$.

We consider its derived version.
Let $X$ be an algebraic variety.
For $a \in D^b(\text{coh}(X))$, 
we take a quasi-isomorphism $a \to I$ to an injective complex, and
we define $R\text{Hom}(F_R,a) = \text{Hom}(F_R,I)$.
Then $R\text{Hom}(F_R,a)$ has a natural right $R$-complex structure.
This complex is well defined in $D^b(R^o)$, where $R^o$ is the opposite ring of $R$, because, 
if $a \to I'$ is another quasi-isomorphism to an injective complex, then there is 
a quasi-isomorphism $\text{Hom}(F_R,I) \to \text{Hom}(F_R,I')$ 
which is uniquely determined up to homotopy.

\begin{Defn}
Let $\{F_i\}_{i=1}^r$ be a simple collection in the category of coherent sheaves $(\text{coh}(X))$ on an 
algebraic variety $X$, 
and let $F_R = \bigoplus_i F_{R,i}$ be an $r$-pointed NC deformation 
over $R \in (\text{Art}_r)$.
Assume that $X$ is quasi-projective and $F_R$ is a perfect complex whose support is projective.

The pair $(F_R,F)$ for $F = \bigoplus_i F_i$ is said to be an 
{\em $r$-pointed relative exceptional object} if 
$R\text{Hom}(F_R,F) \cong R/M$ as right $R$-modules,
i.e., $\dim \text{Hom}(F_R,F_i) = 1$ for all $i$ and $\text{Ext}^p(F_R,F) = 0$ for all $p > 0$.
\end{Defn}

We note that $\text{Hom}(F_{R,i},F_{R,j})$ may not vanish even though $\text{Hom}(F_i,F_j) = 0$ for $i \ne j$.

We consider only those relative exceptional objects which are versal NC deformations
in this paper, but we may consider such objects in other situations.

\begin{Thm}\label{exceptional-SOD}
Let $X$ be a quasi-projective variety, and  
let $(F_R,F)$ with $F = \bigoplus_{i=1}^r F_i$ be an $r$-pointed relative exceptional object in 
$(\text{coh}(X))$ over $R$.
Assume that $F_R$ is a perfect complex whose support is projective.
Let $\langle F_i \rangle_{i=1}^r$ denote the smallest triangulated subcategory of $D^b(\text{coh}(X))$
which contains all $F_i$.
Then there is a semi-orthogonal decomposition 
\[
D^b(\text{coh}(X)) = \langle (\langle F_i \rangle_{i=1}^r)^{\perp}, \langle F_i \rangle_{i=1}^r \rangle.
\]
with an equivalence
\[
\langle F_i \rangle_{i=1}^r \cong D^b(\text{mod-}R).
\]
\end{Thm}

\begin{proof}
We prove the theorem using a down-to-earth method of diagram chasing. 
A more modern proof is presented in the following remark.

First we define the derived dual $F_R^*$ of $F_R$ in the following.
By this step, we will be able to define a functor for the semi-orthogonal decomposition 
by the cone construction. 
$F_R^*$ will be the derived dual in $D^b(\text{coh}(X))$ 
equipped with a uniquely determined right $R$-module structure.

Since $F_R$ is an $R \otimes_{\Bbbk} \mathcal{O}_X$-module which is a perfect complex as an 
$\mathcal{O}_X$-module, there is an exact sequence of left $R \otimes_{\Bbbk} \mathcal{O}_X$-modules
\[
0 \to P_m \to \dots \to P_0 \to F_R \to 0
\]
such that $P_i$ for $0 \le i < m$ are locally free as $R \otimes_{\Bbbk} \mathcal{O}_X$-modules and 
$P_m$ is locally free as an $\mathcal{O}_X$-module.
Then we define $F_R^* = \mathcal{H}om(P_{\bullet},\mathcal{O}_X)$.
It is a complex of right $R \otimes_{\Bbbk} \mathcal{O}_X$-modules.

We prove that $F_R^*$ is well defined as an object in 
$D^b(\mathcal{O}_X)$ with a right $R$-module structure.
Let $P'_{\bullet}$ be another resolution of $F_R$ of length $m'$ as above.
Then by using the lemma below, 
we construct a third resolution $P''_{\bullet}$ of $F_R$ of length $m'' \ge m,m'$ which makes  
the following diagram commutative:
\[
\begin{CD}
0 @>>> P_{m''} @>>> \dots @>>> P_0 @>>> F_R @>>> 0 \\
@AAA @AAA @. @AAA @A=AA \\
0 @>>> P''_{m''} @>>> \dots @>>> P''_0 @>>> F_R @>>> 0 \\
@VVV @VVV @. @VVV @V=VV \\
0 @>>> P'_{m''} @>>> \dots @>>> P'_0 @>>> F_R @>>> 0. 
\end{CD}
\] 
where we set $P_i = 0$ for $i > m$ and $P'_i = 0$ for $i > m'$.

Indeed we construct the $P''_i$ by induction on $i$ as follows.
Assume that the $P''_i$ are already constructed for $i \le i_0$. 
We take a locally free $R \otimes_{\Bbbk} \mathcal{O}_X$-module $P$ which surjects to 
$\text{Ker}(P''_{i_0} \to P''_{i_0-1})$.
Using the lemma, we take a very ample invertible sheaf $L$ and a subspace 
$V \subset H^0(X,L)$ which generates $L$, and let 
$P''_{i_0+1} = V \otimes_{\Bbbk} L^{-1} \otimes P$ with a homomorphism
$P''_{i_0+1} \to P''_{i_0}$ induced by the natural surjective homomorphism 
$V \otimes_{\Bbbk} L^{-1} \to \mathcal{O}_X$.
Then the composite homomorphisms $P''_{i_0+1} \to P''_{i_0} \to P_{i_0}$ and 
$P''_{i_0+1} \to P''_{i_0} \to P'_{i_0}$ factor through $P_{i_0+1}$ and $P'_{i_0+1}$, respectively.

Moreover if there are two vertical morphisms between the same complexes 
which make the diagram commutative, 
then there exists a chain homopoty  
between these morphisms due to the same lemma.

Indeed we construct a chain homotopy $\{h_i: P''_i \to P_{i+1}\}$ between two 
morphisms $\{g_i: P''_i \to P_i\}$ and $\{g'_i: P''_i \to P_i\}$ by induction on $i$ as follows.
Assume that the $h_i$ are already constructed for $i \le i_0$. 
Using the lemma, the difference of $g_{i_0+1} - g'_{i_0+1}$ and the composite homomorphism 
$P''_{i_0+1} \to P''_{i_0} \to P_{i_0+1}$ defined from $h_{i_0}$ factors through $P_{i_0+1}$.

Thus there are uniquely determined morphisms 
$\mathcal{H}om(P_{\bullet},\mathcal{O}_X) \to \mathcal{H}om(P''_{\bullet},\mathcal{O}_X)$
and $\mathcal{H}om(P'_{\bullet},\mathcal{O}_X) \to \mathcal{H}om(P''_{\bullet},\mathcal{O}_X)$
in $D^b(\mathcal{O}_X)$ which are compatible with their right $R$-module structures.
These morphisms are in fact isomorphisms in $D^b(\mathcal{O}_X)$ as is well-known.

We define $F_R^* \boxtimes_R^{\mathbf{L}} F_R = 
\mathcal{H}om(P_{\bullet},\mathcal{O}_X) \boxtimes_R F_R$
as an object in $D^b(\text{coh}(X \times X))$.
The derived tensor product is taken in $D^b(\text{coh}(X \times X))$ and not as $R$-modules.
Here we note that $F_R$ is $R$-flat.
Let $G: D^b(\text{coh}(X)) \to D^b(\text{coh}(X))$ be an integral functor defined by a Fourier-Mukai kernel
$\text{Cone}(F_R^* \boxtimes^{\mathbf{L}}_R F_R \to \Delta_X)$ on $X \times X$.

Since the $P_i$ are locally free, we have
$R\mathcal{H}om(P_{\bullet},a) \cong \mathcal{H}om(P_{\bullet},a)$, hence
\[p_{2*}(p_1^*a \otimes F_R^* \boxtimes_R^{\mathbf{L}} F_R)
= R\Gamma(X,\mathcal{H}om(P_{\bullet},a)) \otimes_R F_R 
= R\text{Hom}(F_R,a) \otimes_R F_R.
\]
Thus we have distinguished triangles
\[
R\text{Hom}(F_R,a) \otimes_R F_R \to a \to G(a) 
\]
for all $a \in D^b(\text{coh}(X))$.
We note that the derived tensor product is taken over the structure sheaves 
but we take usual tensor product over $R$.
This is justified because $F_R$ is flat over $R$.

Since $\text{Hom}(F_R,F_R[p]) = 0$ for all $p \ne 0$,
we have $R\text{Hom}(F_R,F_R) \cong \text{Hom}(F_R,F_R) \cong R$.
Hence we obtain $R\text{Hom}(F_R,G(a)) = 0$ by taking $R\text{Hom}(F_R,\bullet)$ of the 
above triangle. 
Thus we obtain a semi-orthogonal decomposition
\[
D^b(\text{coh}(X)) = \langle \langle F_R \rangle^{\perp}, \langle F_R \rangle \rangle
\]
with an equivalence
\[
\langle F_R \rangle \cong D^b(\text{mod-}R)
\]
by the tilting theory (cf. \cite{Toda-Uehara}~Lemma~3.3), where an equivalence
$\Phi: D^b(\text{mod-}R) \to \langle F_R \rangle$
is given by $\Phi(b) = b \otimes_R F_R$, and its quasi-inverse 
$\Psi: \langle F_R \rangle \to D^b(\text{mod-}R)$ by $\Psi(a) = R\text{Hom}(F_R,a)$.

Finally, since $R\text{Hom}(F_R,F_i) \cong R/M_i$, we have $G(F_i) = 0$.
Therefore $\langle F_R \rangle = \langle F_i \rangle_{i=1}^r$.
\end{proof}

\begin{Lem}
Let $X$ be a quasi-projective variety, let 
\[
\begin{CD}
Q_0 @> {f_1} >> Q_1 @>{f_2}>> \dots @>{f_n}>> Q_n
\end{CD}
\]
be an exact sequence of $R \otimes_{\Bbbk} \mathcal{O}_X$-modules, and let $P$ be a locally free
$R \otimes_{\Bbbk} \mathcal{O}_X$-module.
Then there exists a very ample invertible sheaf $L$ on $X$, without $R$-action, 
which satisfies the following condition:
for any given $R \otimes_{\Bbbk} \mathcal{O}_X$-homomorphism $g_i: L^{-1} \otimes P \to Q_i$ 
such that $f_{i+1} \circ g_i = 0$ for some $1 \le i < n$, 
there exists a $R \otimes_{\Bbbk} \mathcal{O}_X$-homomorphism $h_{i-1}: L^{-1} \otimes P \to Q_{i-1}$ such that 
$g_i = f_i \circ h_{i-1}$.
\end{Lem}

\begin{proof}
It is sufficient to take $L$ such that 
$H^1(X,\mathcal{H}om_{R \otimes_{\Bbbk} \mathcal{O}_X}(P,\text{Ker}(f_i)) \otimes L) = 0$.
\end{proof}

\begin{Rem}\label{Brown}
A referee suggested the following algebraic proof of Theorem~\ref{exceptional-SOD} 
using the Brown representability theorem (\cite{Neeman}~Theorem~4.1):
a triangulated functor between triangulated categories has a right adjoint functor 
if the source category is compactly generated and has all small coproducts (arbitrary direct sums), 
and if the functor respects coproducts. 
This proof shows the power of modern technology compared 
to the down-to-earth old argument above in the same way as in \cite{Neeman}.

We consider unbounded derived categories of quasi-coherent sheaves 
$D(\text{Qcoh}(X))$ and (not necessarily finitely generated) modules $D(\text{Mod-}R)$. 
We define a functor $\tilde{\Phi}: D(\text{Mod-}R) \to D(\text{Qcoh}(X))$ by 
$\tilde{\Phi}(\bullet) = \bullet \otimes_R F_R$.
Since $D(\text{Mod-}R)$ is generated by $D^b(\text{mod-}R)$, it is compactly generated.
Since $\tilde{\Phi}$ commutes with (not necessarily finite) direct sums, it has a right adjoint functor
$\tilde{\Psi}: D(\text{Qcoh}(X)) \to D(\text{Mod-}R)$.
By adjunction, we have
\[
H^p(\tilde{\Psi}(a)) \cong \text{Hom}(R,\tilde{\Psi}(a)[p]) \cong \text{Hom}(\tilde{\Phi}(R),a[p]) 
\cong \text{Hom}(F_R,a[p])
\]
for all $p$.
We note that we do not have to define $R\text{Hom}(F_R,a)$ in this proof.
Since $F_R$ is a perfect complex with proper support, it follows that $\tilde{\Psi}$ induces a 
functor $\Psi: D^b(\text{coh}(X)) \to D^b(\text{mod-}R)$.
Let $\Phi: D^b(\text{mod-}R) \to D^b(\text{coh}(X))$ be the restriction of $\tilde{\Phi}$.

Since $\text{Hom}(F_R,F_i[p]) = 0$ for $p \ne 0$ and all $i$, we have
$\text{Hom}(F_R,F_R[p]) \cong 0$ for $p \ne 0$.
Since $\text{Hom}(F_R,F_R) \cong R$, we have
\[
H^p(\Psi(\Phi(b))) \cong \text{Hom}(R,\Psi(\Phi(b))[p]) \cong \text{Hom}(F_R,b[p] \otimes_R F_R)
\cong H^p(b).
\]
Thus the adjunction morphism $b \to \Psi(\Phi(b))$ is a quasi-isomorphism.
Therefore $\Phi$ is fully faithful, and $\Phi(D^b(\text{mod-}R))$ 
is a right admissible subcategory (\cite{BK}) with 
a semi-orthogonal decomposition as stated in the theorem, where 
$\Phi(D^b(\text{mod-}R))$ is generated by the $F_i = \Phi(R/M_i)$,
since $D^b(\text{mod-}R)$ is generated by the $R/M_i$.
\end{Rem}

We consider some examples which yield relative exceptional objects.

\begin{Expl}\label{quadric2}
Let $X$ be a singular quadric surface in $\mathbf{P}^3$ defined by an equation
$x_1x_2+x_3^2=0$.

Let $P = [1:0:0:0] \in X$ be the vertex.
Then we have a projection $p: X \setminus \{P\} \to \mathbf{P}^1$.
We denote by $\mathcal{O}_X(a)$ the reflexive hull of the invertible sheaf 
$p^*\mathcal{O}_{\mathbf{P}^1}(a)$ for any integer $a$.
$\mathcal{O}_X(2)$ is an invertible sheaf coming from a hyperplane section in $\mathbf{P}^3$,
and we have $\mathcal{O}_X(K_X) \cong \mathcal{O}_X(-4)$.
By the vanishing theorem (\cite{KMM}~Theorem~1.2.5), 
we have $H^p(X,\mathcal{O}_X(a)) = 0$ for $p > 0$ if $a \ge -3$.

Let $F = \mathcal{O}_X(-1)$.
We define an extension
$0 \to F \to G \to F \to 0$ by the following commutative diagram:
\[
\begin{CD}
0 @>>> \mathcal{O}_X(-1) @>>> G @>>> \mathcal{O}_X(-1) @>>> 0\\
@. @V=VV @VVV @VVV @. \\
0 @>>> \mathcal{O}_X(-1) @>>> \mathcal{O}_X^2 @>>> \mathcal{O}_X(1) @>>> 0
\end{CD}
\]
where the right vertical arrow is obtained from an inclusion $\mathcal{O}_X(-2) \to \mathcal{O}_X$
whose cokernel is supported in the smooth locus, 
and the sequence in the second row is exact since $\mathcal{O}_X(D) \cong \mathcal{O}_X(1)$ 
for $D = \{x_1=x_3=0\}$ is generated by global sections $\{1,x_3/x_1\}$. 
It is induced from the following exact sequence on $\mathbf{P}^1$
\[
0 \to \mathcal{O}_{\mathbf{P}^1}(-1) \to \mathcal{O}_{\mathbf{P}^1}^2 \to 
\mathcal{O}_{\mathbf{P}^1}(1) \to 0.
\]
We note that $G$ is a locally free sheaf, hence the extension is non-trivial. 
Moreover there is no more local extension of $G$ by $F$.
Thus the dimension of the local extension at $P$ is $\dim H^0(X,\mathcal{E}xt^1(F,F)) = 1$.

We will prove that there is no more non-trivial extension, 
$G$ is a versal $1$-pointed NC deformation of $F$, 
and that $G$ is a relative exceptional object.
Since $G$ is locally free, it is sufficient to prove that 
$H^p(X, \mathcal{H}om(G,F)) = 0$ for $p > 0$.
We have an exact sequence
\[
0 \to \mathcal{O}_X \to \mathcal{H}om(G,F) \to \mathcal{O}_X \to \mathcal{E}xt^1(F,F) \to 0
\]
where we have $\mathcal{H}om(F,F) \cong \mathcal{O}_X$ because $F$ is a reflexive sheaf of rank $1$.
Let $H = \text{Ker}(\mathcal{O}_X \to \mathcal{E}xt^1(F,F))$.
Then we have $H^p(X,\mathcal{O}_X) = H^p(X,H) = 0$, hence $H^p(X, \mathcal{H}om(G,F)) = 0$ for $p > 0$. 

The base ring of the deformation $G$ is $R = \Bbbk[t]/(t^2)$, 
and $G$ is a relative exceptional object over $R$.
$D^b(\text{coh}(X))$ is generated by $\mathcal{O}_X(a)$ for $-3 \le a \le 0$ (\cite{stacks}~\S 5).
But we have an exact sequence $0 \to \mathcal{O}_X(-3) \to \mathcal{O}_X(-2)^2 \to \mathcal{O}_X(-1) \to 0$.
Hence it is generated by $\mathcal{O}_X(a)$ for $-2 \le a \le 0$.
Therefore we have a full collection of relative exceptional objects $(\mathcal{O}_X(-2), G, \mathcal{O}_X)$, and 
an equivalence  
\[
D^b(\text{coh}(X)) \cong \langle D^b(\Bbbk), D^b(\Bbbk[t]/(t^2)), D^b(\Bbbk) \rangle.
\]
This is a special case of \cite{Kuznetsov}.
But the expressions of the rings seem different, because such expressions are not unique.
The algebra given in loc. cit. seems to be Morita equivalent to $\Bbbk[t]/(t^2)$.
\end{Expl}

\begin{Expl}\label{quadric3}
Let $X$ be a singular quadric hypersurface in $\mathbf{P}^4$ defined by an equation
$x_1x_2+x_3x_4=0$.
It is a cone over $\mathbf{P}^1 \times \mathbf{P}^1$.

Let $P = [1:0:0:0:0] \in X$ be a vertex, and
$p: X \setminus \{P\} \to \mathbf{P}^1 \times \mathbf{P}^1$ a projection. 
We denote by $\mathcal{O}_X(a,b)$ the reflexive hull of an invertible sheaf 
$p^*\mathcal{O}_{\mathbf{P}^1 \times \mathbf{P}^1}(a,b)$ for any $a,b$.
$\mathcal{O}_X(1,1)$ is an invertible sheaf coming from a hyperplane section in $\mathbf{P}^4$,
and we have $\mathcal{O}_X(K_X) \cong \mathcal{O}_X(-3,-3)$.
By the vanishing theorem, we have $H^p(X,\mathcal{O}_X(a,b)) = 0$ for $p > 0$ if $a,b \ge -2$.

Indeed there is a small resolution $f: Y \to X$, a special case of a crepant $\mathbf{Q}$-factorialization, 
and an invertible sheaf $\mathcal{O}_Y(a,b)$ 
such that $\mathcal{O}_Y(a,b) \otimes \mathcal{O}_Y(-K_Y)$ is nef and big and that 
$Rf_*\mathcal{O}_Y(a,b) \cong \mathcal{O}_X(a,b)$.
Thererfore the vanishing follows from \cite{KMM}~Theorem 1.2.3.

Let $F_1 = \mathcal{O}_X(0,-1)$ and $F_2 = \mathcal{O}_X(-1,0)$.
We define an extension
$0 \to F_2 \to G_1 \to F_1 \to 0$ by the following commutative diagram:
\[
\begin{CD}
0 @>>> \mathcal{O}_X(-1,0) @>>> G_1 @>>> \mathcal{O}_X(0,-1) @>>> 0\\
@. @V=VV @VVV @VVV @. \\
0 @>>> \mathcal{O}_X(-1,0) @>>> \mathcal{O}_X^2 @>>> \mathcal{O}_X(1,0) @>>> 0
\end{CD}
\]
where the right vertical arrow is obtained from an inclusion $\mathcal{O}_X(-1,-1) \to \mathcal{O}_X$, 
and the sequence in the second row is exact since $\mathcal{O}_X(D) \cong \mathcal{O}_X(1,0)$ 
for $D = \{x_1=x_3=0\}$ is generated by global sections $\{1,x_4/x_1\}$, where we note that 
$x_4/x_1 = -x_2/x_3$.
It is induced from the following exact sequence on $\mathbf{P}^1 \times \mathbf{P}^1$
\[
0 \to \mathcal{O}_{\mathbf{P}^1 \times \mathbf{P}^1}(-1,0) \to 
\mathcal{O}_{\mathbf{P}^1 \times \mathbf{P}^1}^2 \to 
\mathcal{O}_{\mathbf{P}^1 \times \mathbf{P}^1}(1,0) \to 0.
\]
We note that $G_1$ is a locally free sheaf, hence the extension is non-trivial. 
In a similar way, we construct an extension 
$0 \to F_1 \to G_2 \to F_2 \to 0$ with $G_2$ locally free. 

We will prove that $H^p(X, \mathcal{H}om(G_i,F_j)) = 0$ for $p > 0$ and for all $i,j$.
We may assume that $i = 1$.
We have an exact sequence
\[
0 \to \mathcal{H}om(F_1,F_2) \to \mathcal{H}om(G_1,F_2) \to \mathcal{H}om(F_2,F_2) 
\to \mathcal{E}xt^1(F_1,F_2) \to 0.
\]
Let $H_1 = \text{Ker}(\mathcal{H}om(F_2,F_2) \to \mathcal{E}xt^1(F_1,F_2))$.
Since $\dim H^0(X,\mathcal{E}xt^1(F_1,F_2)) = 1$, we have 
$H^p(X,H_1) = 0$ for $p \ge 0$, hence
$H^p(X, \mathcal{H}om(G_1,F_2)) = 0$ for $p \ge 0$.

On the other hand, the natural homomorphism
$\mathcal{O}_X(1,0) \otimes \mathcal{O}_X(0,-1) \to \mathcal{O}_X(1,-1)$
is surjective, because so is 
$\mathcal{O}_X(1,0) \otimes \mathcal{O}_X(1,0) \to \mathcal{O}_X(2,0)$.
Since there is an exact sequence
\[
0 \to \mathcal{O}_X(0,1) \to G_1^* \to \mathcal{O}_X(1,0) \to 0
\]
$\mathcal{H}om(G_1,F_1) \to \mathcal{O}_X(1,-1)$ is also surjective, and we have an exact sequence
\[
0 \to \mathcal{H}om(F_1,F_1) \to \mathcal{H}om(G_1,F_1) \to \mathcal{H}om(F_2,F_1) \to 0.
\]
Hence $H^p(X, \mathcal{H}om(G_1,F_1)) = 0$ for $p > 0$.

$G_1 \oplus G_2$ is a versal $2$-pointed NC deformation of $F_1 \oplus F_2$
over 
\[
R = \left( \begin{matrix} 
\Bbbk & \Bbbk t \\
\Bbbk t & \Bbbk
\end{matrix} \right) \mod t^2.
\]
$G_1 \oplus G_2$ is a $2$-pointed relative exceptional object over $R$.
We note that $G_1$ and $G_2$ are exceptional objects, but they do not form an exceptional collection,
though there is a semi-orthogonal decomposition with their right orthogonal complement.

Let $f': Y' \to X$ be the blowing up at the vertex, and $E$ the exceptional divisor.
Then there is a $\mathbf{P}^1$-bundle structure $p': Y' \to \mathbf{P}^1 \times \mathbf{P}^1$.
Let $\mathcal{O}_{Y'}(a,b) = (p')^*\mathcal{O}_{\mathbf{P}^1 \times \mathbf{P}^1}(a,b)$.
Since a relative hyperplane class of the bundle $p'$ is given by 
$\mathcal{O}_{Y'}(E) \cong \mathcal{O}_{Y'}(-1,-1)$, 
we deduce that 
$D^b(\text{coh}(Y'))$ is generated by the $\mathcal{O}_{Y'}(a,b)$ for
\[
(a,b) = (-2,-2), (-2,-1), (-1,-2), (-1,-1), (-1,0), (0,-1),(0,0)
\]
by \cite{Orlov}.
Therefore $D^b(\text{coh}(Y))$ and $D^b(\text{coh}(X))$ are also generated 
by the $\mathcal{O}_Y(a,b)$ and the $\mathcal{O}_X(a,b)$ for such $(a,b)$'s, respectively.

We have exact sequences
\[
\begin{split}
&0 \to \mathcal{O}_X(-1,-2) \to \mathcal{O}_X(-1,-1)^2 \to \mathcal{O}_X(-1,0) \to 0 \\
&0 \to \mathcal{O}_X(-2,-1) \to \mathcal{O}_X(-1,-1)^2 \to \mathcal{O}_X(0,-1) \to 0.
\end{split}
\] 
Thus $D^b(\text{coh}(X))$ is generated by 
the $\mathcal{O}_X(a,b)$ for 
\[
(a,b) = (-2,-2), (-1,-1), (-1,0), (0,-1),(0,0).
\]

Therefore we have a full collection of relative exceptional objects 
\[
(\mathcal{O}_X(-2,-2), \mathcal{O}_X(-1,-1), G, \mathcal{O}_X)
\]
for $G = G_1 \oplus G_2$, and 
an equivalence  
\[
D^b(\text{coh}(X)) \cong \langle D^b(\Bbbk),D^b(\Bbbk),D^b(R),D^b(\Bbbk) \rangle.
\]
This is a special case of \cite{Kuznetsov}.
But the expressions of the rings seem different, because such expressions are not unique.
The algebra given in loc. cit. seems to be Morita equivalent to $R$.
\end{Expl}

\begin{Expl}
Let $X = \mathbf{P}(1,1,d)$ be the cone over a rational normal curve of degree $d$.
We have reflexive sheaves of rank one $\mathcal{O}_X(a)$ for integers $a$, and 
$\mathcal{O}_X(K_X) \cong \mathcal{O}_X(-d-2)$.
We consider NC deformations of a sheaf $F = \mathcal{O}_X(-1)$.

Since $\dim H^0(X,\mathcal{O}_X(d-1)) = d$, we have an exact sequence
\[
0 \to \mathcal{O}_X(-1)^{d-1} \to \mathcal{O}_X^d \to \mathcal{O}_X(d-1) \to 0.
\]
This is induced from the following exact sequence on $\mathbf{P}^1$
\[
0 \to \mathcal{O}_{\mathbf{P}^1}(-1)^{d-1} \to \mathcal{O}_{\mathbf{P}^1}^d \to 
\mathcal{O}_{\mathbf{P}^1}(d-1) \to 0.
\]
Let $Z \in \vert \mathcal{O}_X(d) \vert$ be the smooth curve at infinity.
Then we have an exact sequence
\[
0 \to \mathcal{O}_X(-1) \to \mathcal{O}_X(d-1) \to \mathcal{O}_Z(d-1) \to 0.
\]
Since $\dim H^0(Z,\mathcal{O}_Z(d-1)) = d$, there is a surjective homomorphism
$\mathcal{O}_X^d \to \mathcal{O}_Z(d-1)$.
Let $G$ be the kernel.
Then $G$ is a locally free sheaf of rank $d$ on $X$.
Thus we have the following commutative diagram
\[
\begin{CD}
@. @. 0 @. 0 @. \\
@. @. @VVV @VVV @. \\
0 @>>> \mathcal{O}_X(-1)^{d-1} @>>> G @>>> \mathcal{O}_X(-1) @>>> 0 \\
@. @V=VV @VVV @VVV @. \\
0 @>>> \mathcal{O}_X(-1)^{d-1} @>>> \mathcal{O}_X^d @>>> \mathcal{O}_X(d-1) @>>> 0 \\
@. @. @VVV @VVV @. \\
@. @. \mathcal{O}_Z(d-1) @>=>> \mathcal{O}_Z(d-1) @. \\
@. @. @VVV @VVV @. \\
@. @. 0 @. 0 @. 
\end{CD}
\]
where the first horizontal sequence is exact because all other sequences are exact.  
Thus $G$ is an NC deformation of $F= \mathcal{O}_X(-1)$
over $R = \Bbbk[t_1,\dots,t_{d-1}]/(t_1,\dots,t_{d-1})^2$, and 
$\dim H^0(X,\mathcal{E}xt^1(F,F)) = d-1$.
We have an exact sequence
\[
0 \to \mathcal{O}_X \to \mathcal{H}om(G,F) \to \mathcal{O}_X^{d-1} \to \mathcal{E}xt^1(F,F) \to 0.
\]
Hence $H^p(X, \mathcal{H}om(G,F)) = 0$ for $p > 0$ as in the above example.
Therefore $G$ is a versal NC deformation of $F= \mathcal{O}_X(-1)$, and 
$G$ is a relative exceptional object over $R$.

By the vanishing theorem, we have $H^p(X,\mathcal{O}_X(a)) = 0$ for $p > 0$ and $a \ge -d-1$.
By \cite{stacks}, $D^b(\text{coh}(X))$ is generated by the $\mathcal{O}_X(a)$ for $-d-1 \le a \le 0$.
But there are exact sequences
$0 \to \mathcal{O}_X(-d-1)^{d-i} \to \mathcal{O}_X(-d)^{d-i+1} \to \mathcal{O}_X(-i) \to 0$
for $1 \le i \le d-1$, hence it is generated only by $\mathcal{O}_X(a)$ for $a = -d,-1,0$.
We have a full collection of relative exceptional objects 
$(\mathcal{O}_X(-d), G, \mathcal{O}_X)$, and 
an equivalence  
\[
D^b(\text{coh}(X)) \cong \langle D^b(\Bbbk),D^b(R),D^b(\Bbbk) \rangle.
\]
\end{Expl}

\begin{Expl}
Let $X = \mathbf{P}(1,2,3)$ be a weighted projective surface.
$X$ has two singular points $P$ and $Q$ which are Du Val singularities of types 
$A_1$ and $A_2$, respectively.
We have $\mathcal{O}_X(K_X) \cong \mathcal{O}_X(-6)$, hence
$H^0(X,\mathcal{O}_X(-i)) = 0$ for $0 < i$, and 
$H^p(X,\mathcal{O}_X(-i)) = 0$ for $p > 0$ and $i < 6$.
We write $F_i = \mathcal{O}_X(-i)$.

\vskip 1pc

First we consider NC deformations of a reflexive sheaf of rank one $F_1 = \mathcal{O}_X(-1)$.
In this example, the non-commutative deformations of $F_1$ do not terminate after finite steps,
though commutative deformations do.

Let us calculate local extensions of $F_1$ at the singular points.
The singular point $P$ is a quotient singularity of type 
$\frac 12(1,1)$.
Then it is already known by the previous example that 
$\mathcal{E}xt^1(F_1,F_1)_P \cong \Bbbk$, and the versal NC local deformation 
has the base ring $\Bbbk[s]/(s^2)$.

The singular point $Q$ is a quotient singularity of type $\frac 13(1,2)$.
Let a cyclic group $\mathbf{Z}/(3)$ act on $\Bbbk[x,y]$ with weights $(1,2)$, and
let $A \subset \Bbbk[x,y]$ be the invariant subring.
We may assume that the sheaf $F_1$ at $Q$ is represented by the ideal $(x) \cap A = (x^3,xy)$ in $A$, 
and $F_2$ at $Q$ by $(y) \cap A = (xy,y^3)$ or $(x^2) \cap A = (x^3, x^2y^2)$.
There are exact sequences 
\[
\begin{split}
&0 \to (x) \cap A \to A \oplus ((x^2) \cap A) \to (x) \cap A \to 0 \\
&0 \to (x) \cap A \to A^2 \to (y) \cap A \to 0
\end{split}
\]
where in the first sequence, the map $A \to (x) \cap A$ is given by $1 \mapsto xy$,
the map $(x^2) \cap A = (x^3,x^2y^2) \to (x) \cap A = (x^3,xy)$ and the map $(x) \cap A \to A$ are 
natural injections, and the map $(x) \cap A \to (x^2) \cap A$ is induced from $1 \mapsto -xy$.
In the second sequence, the map $A^2 \to (y) \cap A$ is given by the generators $(xy,y^3)$, and
the map $(x) \cap A \to A^2$ is given in the following way:
the map to the first entry of $A$ is the composition of the isomorphism given by $x \mapsto -y^2$ 
and the natural injection $(x) \cap A \cong (y^2) \cap A \to A$, while the second is the natural injection 
$(x) \cap A \to A$.

We claim that the versal NC deformation of the ideal $(x) \cap A$ at $Q$ has the base ring $\Bbbk[t]/(t^3)$.
For this purpose, it is sufficient to prove that $\mathcal{E}xt^1(F_1,F_1)_Q \cong \Bbbk$.
From an exact sequence 
\[
0 \to (y) \cap A \to A^2 \to (x) \cap A \to 0
\]
we obtain an exact sequence
\[
0 \to A \to ((x) \cap A) \oplus ((y^2) \cap A) \to (y) \cap A \to \mathcal{E}xt^1(F_1,F_1)_Q \to 0
\]
where the first arrow is given by $1 \mapsto (xy,x^2y^2)$ and the second induced by $1 \mapsto -xy$ 
and a natural inclusion.
Then the cockerel of the second arrow is generated by the image of $xy$, hence $1$-dimensional.

\vskip 1pc

We prove that a sequence of iterated non-trivial extensions $G_1^n$ of $F_1$ 
never become locally free for any $n$ by induction on $n$.
If $G_1^n$ is not locally free at $P$, then we take an extension
\[
0 \to F_1 \to G_1^{n+1} \to G_1^n \to 0
\]
which induces a non-trivial extension at $P$ and a trivial extension at $Q$.
Then $G_1^{n+1}$ is not locally free at $Q$.
The same argument works if we interchange $P$ and $Q$. 
Therefore we proved our assertion. 

Indeed we could prove that the versal NC deformation has a base ring 
$\Bbbk\langle s,t \rangle/(s^2,t^3)$, which is infinite dimensional,
while its maximal abelian quotient $\Bbbk[s,t]/(s^2,t^3)$ is finite dimensional. 

\vskip 1pc

Next we consider $1$-pointed NC deformations of reflexive sheaves $F_2 = \mathcal{O}_X(-2)$ and
$F_3 = \mathcal{O}_X(-3)$.
Then the results are better, because $F_2$ (resp. $F_3$) is locally free at $P$ (resp. $Q$).
We claim the following: $F_2$ (resp. $F_3$) has a versal NC deformation 
$G_2$ (resp. $G_3$) over $\Bbbk[t]/(t^3)$ (resp. $\Bbbk[s]/(s^2)$) which is a locally 
free sheaf of rank $3$ (resp. $2$), and they are relative exceptional objects.

As for $F_3$, we can prove that $\text{Ext}^p(G_3,F_3) = 0$ for $p > 0$ in the same way as in 
Example~\ref{quadric2}.
We consider $F_2$.
Since $H^p(X, \mathcal{H}om(F_2,F_2)) = 0$ for $p > 0$, we have
$\text{Ext}^1(F_2,F_2) \cong H^0(X,\mathcal{E}xt^1(F_2,F_2)) \cong \Bbbk$, since 
we already proved $\mathcal{E}xt^1(F_1,F_1))_Q \cong \Bbbk$.
Therefore we have a non-trivial extension $0 \to F_2 \to G'_2 \to F_2 \to 0$.
We consider an exact sequence
\[
0 \to \mathcal{H}om(F_2,F_2) \to \mathcal{H}om(G'_2,F_2) \to \mathcal{H}om(F_2,F_2) 
\to \mathcal{E}xt^1(F_2,F_2).
\]
Since $G'_2$ is a non-trivial extension also locally at $Q$, the last arrow is non-zero.
Let $H$ be the kernel of the last arrow.
Then we have $H^p(X,H) = 0$ for $p > 0$, hence $H^p(X,\mathcal{H}om(G'_2,F_2)) = 0$ for $p > 0$.
Therefore $\text{Ext}^1(G'_2,F_2) \cong H^0(X,\mathcal{E}xt^1(G'_2,F_2)) \cong \Bbbk$, and 
we have a non-trivial extension $0 \to F_2 \to G_2 \to G'_2 \to 0$.
In the same way, we obtain $H^p(X,\mathcal{H}om(G_2,F_2)) = 0$ for $p > 0$.
Since $G_2$ is locally free, this is the desired result. 

\vskip 1pc

We claim that the category $D^b(\text{coh}(X))$ is generated by the reflexive sheaves
$\mathcal{O}_X(-m)$ for $m = 0,2,3$.
It is already known that it is generated by the $\mathcal{O}_X(-m)$ for $0 \le m \le 5$ (\cite{stacks}~\S 5).
We check that $\mathcal{O}_X(-m)$ for $m = 1,4,5$ are generated by others.

Let $D_1$ be the coordinate divisor such that $\mathcal{O}_X(D_1) \cong \mathcal{O}_X(1)$.
We consider exact sequences 
\[
0 \to \mathcal{O}_X(-m) \to \mathcal{O}_X \to \mathcal{O}_{mD_1} \to 0
\]
for $1 \le m \le 5$, where the third terms are defined as cokernels.
Since $H^p(X,\mathcal{O}_X) \cong \Bbbk$ for $p=0$ and $\cong 0$ for $p > 0$, and 
$H^p(X,\mathcal{O}_X(-m)) \cong 0$ for all $p$ and $1 \le m \le 5$, we deduce that 
$H^0(X,\mathcal{O}_{mD_1}) \cong \Bbbk$ and $H^1(X,\mathcal{O}_{mD_1}) \cong 0$.
Therefore we have an exact sequence
\[
0 \to \mathcal{O}_{D_1}(-1) \to \mathcal{O}_{mD_1} \to \mathcal{O}_{(m-1)D_1} \to 0
\]
for $0 < m \le 5$.
Indeed the first term is isomorphic to $\text{Coker}(i)$ in 
following commutative diagram of exact sequences:
\[
\begin{CD}
0 @>>> \mathcal{O}_X(-m) @>>> \mathcal{O}_X @>>> \mathcal{O}_{mD_1} @>>> 0 \\
@. @ViVV @V=VV @VVV @. \\
0 @>>> \mathcal{O}_X(-m+1) @>>> \mathcal{O}_X @>>> \mathcal{O}_{(m-1)D_1} @>>> 0
\end{CD}
\]
where $\text{Coker}(i)$ is an invertible sheaf on $D_1$, because
$X$ has only quotient singularities and $\text{Coker}(i)$ is locally the sheaf of invariants of the 
corresponding sheaf on the covering.
Therefore $\mathcal{O}_{D_1}(-1)$ is expressed by the $\mathcal{O}_X(-m)$ for $m = 0,2,3$, and 
so are the $\mathcal{O}_X(-m)$ for $m=1,4,5$.

\vskip 1pc

In conclusion, we have a full collection of relative exceptional objects 
$(G_3, G_2, \mathcal{O}_X)$, and 
an equivalence  
\[
D^b(\text{coh}(X)) \cong \langle D^b(\Bbbk[s]/(s^2)),D^b(\Bbbk[t]/(t^3)),D^b(\Bbbk) \rangle.
\]

By a similar method, the author expects that the following can be proved.
Let $X = \mathbf{P}(1,a,b)$ be a weighted projective plane for coprime positive integers $a,b$ with $a < b$.
We consider $1$-pointed NC deformations of $F_a = \mathcal{O}_X(-a)$ and $F_b = \mathcal{O}_X(-b)$.
We could prove that there exist versal deformations $G_a$ and $G_b$ of $F_a$ and $F_b$, respectively, 
which are locally free and relative exceptional objects.
Moreover there is a semi-orthogonal decomposition
$D^b(\text{coh}(X)) = \langle G_b,G_a,\mathcal{O}_X \rangle$.
\end{Expl}

\section{$r$-pointed relative spherical objects on Calabi-Yau threefolds}

We define relative spherical objects after \cite{Toda-twist} and \cite{Ann-Log}, and prove that 
a versal multi-pointed NC deformation of a simple collection on a Calabi-Yau 3-fold yields a 
relative spherical object if the deformations stops after finitely many non-trivial extensions 
and if one more condition holds.

\begin{Defn}
Let $X$ be a smooth projective variety of dimension $n \ge 2$, 
let $\{F_i\}_{i=1}^r$ be a simple collection in $(\text{coh}(X))$,
and let $F_R = \bigoplus_i F_{R,i}$ be an $r$-pointed NC deformation of $\{F_i\}$ over $R \in (\text{Art}_r)$.
The pair $(F_R,F)$ for $F= \bigoplus_{i=1}^r F_i$ is said to be an 
{\em $r$-pointed relative $n$-spherical object} over $R$ if the following conditions are satisfied:

\begin{enumerate}

\item There exists a permutation $\sigma$ of $r$ elements such that
\[
\text{Hom}(F_R,F_i[p]) \cong \begin{cases} R/M_i  \qquad &p = 0 \\
R/M_{\sigma(i)} \qquad &p = n \\
0 \qquad &p \ne 0,n
\end{cases}
\]
as right $R$-modules for all $i$.

\item $F \otimes \omega_X \cong F$.
\end{enumerate}
\end{Defn}

More generally, for a triangulated category with a Serre functor $S$, 
the second condition can be replaced by $S(F) \cong F[n]$. 

The following lemma shows that the base ring $R$ of an $r$-pointed relative $n$-spherical object
is a NC Gorenstein artin algebra (\cite{AR}):

\begin{Lem}
Let $(F_R,F)$ be an $r$-pointed relative spherical object over $R$.
Then $R^* = \text{Hom}_{\Bbbk}(R,\Bbbk)$ is a free right $R$-module of rank $1$.
\end{Lem}

\begin{proof}
Since $\dim \text{Hom}(F_R,F_i) = \dim \text{Hom}(F_i,F_R) = 1$ by the Serre duality, we can define  
$s_i \in R = \text{Hom}(F_R,F_R)$ 
as a composition of non-zero homomorphisms $F_R \to F_i \to F_R$ up to a constant.
Let $\phi \in R^*$ be a homomorphism $R \to \Bbbk$ such that $\phi(s_i)=1$ for all $i$.

We will prove that $\phi$ generates $R^*$ as a right $R$-module.
Let $I = \{s \in R \mid \phi s = 0\}$ be the annihilator ideal of $\phi$.
If $I \ne 0$, then there is a socle; there exist $i$ and $0 \ne s \in I$ 
such that $M_is = 0$ for the $i$-th maximal ideal $M_i$ of $R$.
We know that such non-zero $s \in R$ that $M_is = 0$ is unique up to a constant for a fixed $i$, 
because $\dim \text{Hom}(F_i,F_R) = 1$.
It follows that $s = cs_i$ for $0 \ne c \in \Bbbk$.
Then $0 = \phi s(1) = \phi(cs_i) =  c$, a contradiction.
Hence $I = 0$.
Since $\dim_{\Bbbk} R < \infty$, we conclude the proof.
\end{proof}

The conclusion of the lemma is equivalent to saying that the socle of $R$ is isomorphic to 
$\bigoplus R/M_i$.

We have the duality theorem:

\begin{Cor}\label{AR}
Let $M$ be a finitely generated right $R$-module.
Then there is a natural isomorphism
\[
\text{Hom}_{\Bbbk}(M,\Bbbk) \cong \text{Hom}_R(M,R^*)
\]
with $R^* \cong R$.
\end{Cor}

\begin{proof}
We define a natural isomorphism
\[
\text{Hom}_{\Bbbk}(M,\Bbbk) \cong \text{Hom}_R(M,\text{Hom}_{\Bbbk}(R,\Bbbk))
\]
as follows.
For a $\Bbbk$-homomorphism $f: M \to \Bbbk$, we define a right $R$-homomorphism 
$g: M \to \text{Hom}_{\Bbbk}(R,\Bbbk)$ by $g(m,r) = f(mr)$.
$g$ is compatible with the right $R$-action: $g(ms,r) = g(m,sr)$.
For $g: M \to \text{Hom}_{\Bbbk}(R,\Bbbk)$, we define $f(m) = g(m,1)$.
They are inverses each other.
\end{proof}

We simply write
$D(R) = D^b(\text{mod-}R)$, $D(X) = D^b(\text{coh}(X))$, etc. in the following. 
We consider the following diagram of \lq\lq spaces''
\[
\begin{CD}
[R] @<{p_1}<< [R] \times X @>{p_2}>> X
\end{CD}
\] 
where $[R]$ is the imaginary \lq\lq space'' corresponding to the ring $R$ 
and $[R] = \text{Spec}(R)$ if $R$ is commutative.

We define 
\[
\begin{split}
&p_1^*(\bullet) = \bullet \otimes_{\Bbbk} \mathcal{O}_X, 
p_1^!(\bullet) = \bullet \otimes_{\Bbbk} \omega_X[n], 
p_{1*}(\bullet) = R\Gamma(X,\bullet), \\
&p_2^*(\bullet) = R \otimes_{\Bbbk} \bullet, 
p_2^!(\bullet) = \text{Hom}_{\Bbbk}(R,\bullet), 
p_{2*}(\bullet) = \bullet. 
\end{split}
\]
We define an exact functor
$\Phi: D(R) \to D(X)$ by 
\[
\Phi(a) = a \otimes_R F_R = p_{2*}(p_1^*a \otimes_{R \otimes \mathcal{O}_X} F_R).
\]
We recall that the derived dual $F_R^*$ is defined in the proof of Theorem 5.2.

\begin{Lem}
The functor $\Phi$ has right and left adjoint functors $\Psi_R: D(X) \to D(R)$ and $\Psi_L: D(X) \to D(R)$
defined as follows:
\[ 
\begin{split}
&\Psi_R(b) = R\text{Hom}_{\mathcal{O}_X}(F_R,b) = 
p_{1*}(F_R^* \otimes_{R \otimes \mathcal{O}_X} p_2^!b) \\
&\Psi_L(b) = p_{1*}(F_R^* \otimes p_2^*\omega_X[n] \otimes p_2^*b) \cong \Psi_R(b)[n].
\end{split}
\]
\end{Lem}

\begin{proof}
On the imaginary spaces, we have
\[
\begin{split}
&\text{Hom}_{D(X)}(\Phi(a),b) 
\cong \text{Hom}_{D([R] \times X)}(p_1^*a \otimes_{R \otimes \mathcal{O}_X} F_R, p_2^!b) \\
&\cong \text{Hom}_{D([R] \times X)}(p_1^*a, F_R^* \otimes_{R \otimes \mathcal{O}_X} p_2^!b)
\cong \text{Hom}_{D(R)}(a,\Psi_R(b)).
\end{split}
\]
Or we have
\[
\begin{split}
&\text{Hom}_{\mathcal{O}_X}(\Phi(a),b) 
\cong \text{Hom}_{R \otimes \mathcal{O}_X}((a \otimes_{\Bbbk} \mathcal{O}_X) 
\otimes_{R \otimes \mathcal{O}_X} F_R,  \text{Hom}_{\Bbbk}(R,b)) \\
&\cong \text{Hom}_{R \otimes \mathcal{O}_X}(a \otimes_{\Bbbk} \mathcal{O}_X, 
F_R^* \otimes_{R \otimes \mathcal{O}_X} \text{Hom}_{\Bbbk}(R,b))
\cong \text{Hom}_R(a,\Psi_R(b)).
\end{split}
\]

Since $p_1^!(a) = p_1^*(a) \otimes p_2^*\omega_X[n]$, we have
\[
\begin{split}
&\text{Hom}_{D(X)}(b,\Phi(a)) \cong \text{Hom}_{D([R] \times X)}(p_2^*b,
p_1^*a \otimes_{R \otimes \mathcal{O}_X} F_R) \\
&\cong \text{Hom}_{D([R] \times X)}((F_R^* \otimes p_2^*\omega_X[n] \otimes p_2^*b, p_1^!a)
\cong \text{Hom}_{D(R)}(\Psi_L(b),a).
\end{split}
\]
Or we have  
\[
\begin{split}
&\text{Hom}_{\mathcal{O}_X}(b,\Phi(a)) 
\cong \text{Hom}_{R \otimes \mathcal{O}_X}(R \otimes_{\Bbbk} b,
(a \otimes_{\Bbbk} \mathcal{O}_X) \otimes_{R \otimes \mathcal{O}_X} F_R) \\
&\cong \text{Hom}_{R \otimes \mathcal{O}_X}(F_R^* \otimes \omega_X[n] \otimes b, 
a \otimes_{\Bbbk} \omega_X[n]) \\
&\cong \text{Hom}_R(R\Gamma(X,F_R^* \otimes \omega_X[n] \otimes b), a)
\cong \text{Hom}_R(\Psi_L(b),a).
\end{split}
\]
where we used the Serre duality
\[
\text{Hom}_{\mathcal{O}_X}(c, \omega_X[n])
\cong \text{Hom}_{\Bbbk}(R\Gamma(X,c),\Bbbk).
\]

Since $p_2^* \cong p_2^!$ by Corollary~\ref{AR} and since 
$\omega_X \otimes F_R \cong F_R$, we obtain the last isomorphism.
\end{proof}

We note that the right adjoint functor $\Psi_R$ is already constructed in \S 5.
It can be obtained without using the derived dual $F_R^*$ by Remark~\ref{Brown},
but we need for the following Theorem~\ref{spherical functor}
the additional property that these functors are of Fourier-Mukai type, 
because these functors should be lifted to the enhancement of the triangulated categories (\cite{Ann-Log}).

We define the twist functor $T: D(X) \to D(X)$ associated to $\Phi$ to be the functor 
corresponding to the Fourier-Mukai kernel
$\text{cone}(F_R^* \boxtimes F_R \to \mathcal{O}_{\Delta_X}) \in D(X \times X)$, where 
$\Delta \subset X \times X$ is the diagonal.  
Then we have a distinguished triangle of functors
\[
T[-1] \to \Phi \circ \Psi_R \to \text{Id}_{D(X)} \to T
\]

\begin{Thm}\label{spherical functor}
The functor $\Phi: D^b(\text{mod-}R) \to D^b(\text{coh}(X))$ given by
$\Phi(a) = a \otimes_R F_R$ is a spherical functor.
In particular, the twist functor $T$ which induces the following distinguished triangles :
\[
T(a)[-1] \to R\text{Hom}(F_R,a) \otimes_R F_R \to a \to T(a) 
\]
for $a \in D^b(\text{coh}(X))$ is an auto-equivalence of $D^b(\text{coh}(X))$.
\end{Thm}

We note that, because $F_R$ is flat over $R$, the tensor product $R\text{Hom}(F_R,a) \otimes_R F_R$ 
is the same as the derived tensor product, and is bounded.

\begin{proof}
We have to check the following conditions (\cite{Ann-Log}, \cite{Meachan}):
\begin{itemize}
\item The cotwist functor $C$ defined by a distinguished triangle
\[
C \to \text{Id}_{D(R)} \to \Psi_R \circ \Phi \to C[1]
\]
is an auto-equivalence of $D^b(\text{mod-}R)$.

\item $\Psi_R \cong C \circ \Psi_L[1]$.
\end{itemize}

$D^b(\text{mod-}R)$ is spanned by $R$.
We have 
\[
\Psi_R(\Phi(R)) \cong \Psi_R(F_R) \cong R\text{Hom}(F_R,F_R).
\]
Therefore we have a distinguished triangle
\[
C(R) \to R \to R\text{Hom}(F_R,F_R) \to C(R)[1]
\]
hence $C(R) \cong R[-1-n]$, and $C$ is an auto-equivalence.

We have $C(\Psi_L(F_R)) \cong C(R\text{Hom}(F_R,F_R))[n]
\cong R\text{Hom}(F_R,F_R)[-1] \cong \Psi_R(F_R)[-1]$.
Thus we confirmed the conditions.
\end{proof}

\begin{Thm}
Let $\{F_i\}_{i=1}^r$ be a simple collection of coherent sheaves on a smooth projective variety $X$ of 
dimension $3$ such that $F \otimes \omega_X \cong F$ for $F = \bigoplus F_i$.
Assume that the versal $r$-pointed NC deformation $F_R$ is obtained by a finite sequence of iterated 
non-trivial $r$-pointed extensions. 
Assume moreover that $\text{Hom}(F_i,F_R) \ne 0$ for all $i$.
Then $(F_R,F)$ is relatively $3$-spherical over $R$.
\end{Thm}

We note that the last condition holds trivially if $r=1$. 

\begin{proof}
We have already $\text{Ext}^1(F_R,F_i) = 0$ for all $i$, because $F_R$ is versal.
Then we have $\text{Ext}^1(F_R,G) = 0$ for any extension $G$ of the $F_i$. 
We have an exact sequence 
\[
0 \to F_i \to F_R \to G_i \to 0
\]
for some $G_i$ for each $i$.
Thus
\[
\text{Ext}^1(F_R,G_i) \to \text{Ext}^2(F_R,F_i) \to \text{Ext}^2(F_R,F_R) 
\]
Since the last term is dual to $\text{Ext}^1(F_R,F_R)=0$, we conclude that 
$\text{Ext}^2(F_R,F_i) = 0$.  

Let $m_i$ be the number of appearances of $F_i$ in the iterated extension $F_R$.
Then
\[
\begin{split}
&\sum_i m_i = \dim \text{Hom}(F_R,F_R) = \dim \text{Ext}^3(F_R,F_R) \\
&= \sum_i m_i \dim \text{Ext}^3(F_R,F_i) = \sum_i m_i \dim \text{Hom}(F_i,F_R).
\end{split}
\]
Since $\text{Hom}(F_i,F_R) \ne 0$ for all $i$, it follows that 
$\dim \text{Hom}(F_i,F_R) = 1$ for all $i$.
Therefore we have $\dim \text{Ext}^3(F_R,F_i) = 1$ for all $i$, and we conclude the proof by the following lemma.
\end{proof}

\begin{Lem}
Let $F_R$ be an $r$-pointed NC deformation of a simple collection $\{F_i\}$ over $R \in (\text{Art}_r)$.
Assume that $\dim \text{Hom}(F_i,F_R) = 1$ for all $i$.
Then there exists a permutation $\sigma$ of $r$ elements such that 
$\text{Hom}(F_i,F_R) \cong R/M_{\sigma(i)}$ as left $R$-modules for all $i$. 
\end{Lem}

\begin{proof}
As left $R$-modules, we have 
$\text{Hom}(F_i,F_R) = R/M_j$ for some $j = j(i)$.
Then we have $\dim \text{Hom}(F_i,F_{R,k}) = \delta_{jk}$.
On the other hand, for each $k$, there is at least one $i$ such that $\text{Hom}(F_i,F_{R,k}) \ne 0$.
Therefore we have a one to one correspondence.
\end{proof}

We consider some examples.

\begin{Expl}
Let $f: X \to Y$ be a projective birational morphism from a smooth variety of dimension 3 to a normal variety
over $\Bbbk = \mathbf{C}$. 
We assume that  $K_X$ is relatively trivial and the exceptional locus of $f$ is $1$-dimensional.
In this case, $Y$ has only terminal Gorenstein singularities, 
and the irreducible components $C_i$ ($i = 1,\dots, r$) 
of the exceptional locus are smooth rational curves intersecting each other transversally, 
because $R^1f_*\mathcal{O}_X = 0$ by the vanishing theorem (\cite{KMM})
applied to an invertible sheaf $\mathcal{O}_X \cong \mathcal{O}_X(K_X)$.

Let $F_i = \mathcal{O}_{C_i}(-1)$.
Then $\{F_i\}_{r=1}^r$ is a simple collection.
But there are many other simple collections on $X$.
For example, for any disjoint subsets $I_j$ of the set of indexes $\{1,\dots,r\}$, 
if the $D_j = \bigcup_{i \in I_j} C_i$
are connected, then $\{\mathcal{O}_{D_j}\}$ is a simple collection.
Let $l_i$ be the length of $C_i$, i.e., the length of the scheme theoretic fiber over a singular point of $Y$
at the generic point of $C_i$,
and take an integer $k_i$ for each $i$ such that $0 <  k_i \le l_i$.
Then a collection consisting of fat curves $\{\mathcal{O}_{k_iC_i}\}$ is simple 
if $\text{End}(\mathcal{O}_{k_iC_i}) = \Bbbk$ for all $i$.
It is interesting to know whether these collections satisfy conditions of the above theorem 
yielding spherical objects, and what are their relationships.
\end{Expl}

\begin{Expl}
Let $Y$ be a hypersurface in $\Bbbk^4$ defined by an equation $xy - zw(z+w)=0$.
It has an isolated singularity at the origin.

We define a resolution of singularities $f: X \to Y$ in the following way.
$X$ is constructed by gluing three affine spaces $X = \bigcup_{i=1}^3 U_i$. 
$U_1$, $U_2$ and $U_3$ are isomorphic to $\mathbf{A}^3$ with coordinates $(x,z',w)$, $(x',z,w')$ and $(x'',y,z)$,
respectively, and $f$ is given by 
\[
\begin{split}
&U_1: \quad (x,y,z,w) \mapsto (x, z'w(xz'+w), xz',w) \\
&U_2: \quad (x,y,z,w) \mapsto (x'z,w'(z+x'w'),z,x'w') \\
&U_3: \quad (x,y,z,w) \mapsto (x''z(x''y -z),y,z,x''y -z) 
\end{split}
\]
where we considered $z'=z/x$, $x'=x/z$, $w'=w/x'=zw/x$ and $x''=x'/w=x/zw$.
$f$ is the composition of a blowing up along the ideal $(x,z)$ of a prime Weil divisor followed by another along $(x',w)$, the ideal corresponding
to the strict transform of a prime divisor defined by $(x,w)$.

\vskip 1pc

The exceptional locus of $f$ consists of two smooth rational curves $C_1 \cup C_2$
which intersect transversally.
$C_1$ is defined on $U_1$ by an ideal $(x,w)$, and on $U_2$ by $(z,w')$.
$C_2$ is defined on $U_2$ by an ideal $(x',z)$, and on $U_3$ by $(y,z)$. 

Let $F_i = \mathcal{O}_{C_i}(-1)$ for $i=1,2$.
Then $\{F_1,F_2\}$ is a simple collection.
We consider its NC deformations.

\vskip 1pc

The conormal bundles $N^*_{C_i/X}$ of the $C_i$ for $i=1,2$ are calculated as follows.
$C_1$ has coordinates $z'$ on $U_1$ and $x'$ on $U_2$, and they are related by $z' = (x')^{-1}$.
The generating sections of $N^*_{C_1/X}$ are transformed as $x \mapsto x'z$ and $w \mapsto x'w'$.
Therefore $N^*_{C_1/X} \cong \mathcal{O}_{\mathbf{P}^1}(1)^2$.

$C_2$ has coordinates $w'$ on $U_2$ and $x''$ on $U_3$, and they are related by $w' = (x'')^{-1}$.
The generating sections of $N^*_{C_2/X}$ are transformed as $x' \mapsto x''(x''y-z)$ and $z+w'x' \mapsto x''y$.
Therefore $N^*_{C_2/X} \cong \mathcal{O}_{\mathbf{P}^1}(1)^2$.

\vskip 1pc

The union $\Theta = C_1 \cup C_2$ is defined by ideals $(x,w)$ on $U_1$, $(x'w',z)$ on $U_2$, and $(y,z)$ on $U_3$. 
We denote by $\mathcal{O}_{\Theta}(a,b)$ an invertible sheaf on $\Theta$ whose restrictions to the $C_i$ for $i=1,2$
have degrees $a$ and $b$.
Let $G_1 \cong \mathcal{O}_{\Theta}(-1,0)$ and $G_2 \cong \mathcal{O}_{\Theta}(0,-1)$. 
Then we have non-trivial extensions
\[
0 \to F_{i'} \to G_i \to F_i \to 0
\]
where $i + i' = 3$.

\vskip 1pc

We calculate the normal bundle $N^*_{\Theta/X} = I_{\Theta}/I_{\Theta}^2$ of $\Theta$.
It is generated by linearly independent sections $s_j, t_j$ on $U_j$ which are defined as follows:
$s_1 = x$ and $t_1 = w + z'x$ on $U_1$, $s_2= z$ and $t_2 = z + x'w'$ on $U_2$, 
and $s_3= z$ and $t_3 = y$ on $U_3$.
We have $s_1 = x's_2$ and $t_1 = t_2$ on $U_1 \cap U_2$, 
and $s_2 = s_3$ and $t_2 = x''t_3$ on $U_2 \cap U_3$.
Therefore we have 
\[
N^*_{\Theta/X} \cong \mathcal{O}_{\Theta}(1,0) \oplus \mathcal{O}_{\Theta}(0,1). 
\]

Let $g: \Theta \to X$ be the embedding.
Since $\Theta$ is a locally complete intersection, we obtain the following by a Koszul resolution:
\[
\begin{split}
&g^*g_*\mathcal{O}_{\Theta} \cong \mathcal{O}_{\Theta} \oplus N^*_{\Theta/X}[1] 
\oplus \text{det}(N^*_{\Theta/X})[2] \\
&\cong \mathcal{O}_{\Theta} \oplus (\mathcal{O}_{\Theta}(1,0) \oplus \mathcal{O}_{\Theta}(0,1))[1] 
\oplus \mathcal{O}_{\Theta}(1,1)[2]
\end{split}
\]
where the direct sum decomposition to cohomologies is a consequence of the fact 
that $\dim \Theta = 1$.
Therefore we have 
\[
\begin{split}
&\text{Ext}^1_X(G_1,F_1) \cong \text{Hom}_{\Theta}
(\mathcal{O}_{\Theta}(0,0) \oplus \mathcal{O}_{\Theta}(-1,1),F_1) \cong \Bbbk \\
&\text{Ext}^1_X(G_2,F_2) \cong \text{Hom}_{\Theta}
(\mathcal{O}_{\Theta}(1,-1) \oplus \mathcal{O}_{\Theta}(0,0),F_2) \cong \Bbbk.
\end{split}
\]
Hence we have non-trivial extensions
\[
0 \to F_i \to F_{R,i} \to G_i \to 0
\]
for $i = 1,2$.

\vskip 1pc

The extension $F_{R,1}$ is induced from the surjection $\mathcal{O}_{\Theta}(-1,1) \to F_1$. 
Hence $F_{R,1}$ is an invertible sheaf of degrees $(-1,0)$ on 
a subscheme $\Theta_1$ of $X$ defined by ideals 
$(s_1,xt_1,wt_1) = (x,w^2)$ on $U_1$, $(s_2,zt_2,w't_2) = (z,x'(w')^2)$ on $U_2$, 
and $(s_3,t_3) = (z,y)$ on $U_3$.
$\Theta_1$ is not reduced along $C_1$, but is still locally complete intersection.

The extension $F_{R,2}$ is induced from the surjection $\mathcal{O}_{\Theta}(1,-1) \to F_2$. 
Hence $F_{R,2}$ is an invertible sheaf of degrees $(0,-1)$ on 
a subscheme $\Theta_2$ of $X$ defined by ideals 
$(s_1,t_1) = (x,w)$ on $U_1$, $(s_2x',s_2z,t_2) = (x'z,z^2,z+x'w') = (x'z,z+x'w')=(z+x'w',(x')^2w')$ on $U_2$, 
and $(s_3y,s_3z,t_3) = (z^2,y)$ on $U_3$.
$\Theta_2$ is not reduced along $C_2$, but is still locally complete intersection. 

\vskip 1pc

We calculate the conormal bundles $N^*_{\Theta_i/X} = I_{\Theta_i}/I_{\Theta_i}^2$ of the fat curves 
$\Theta_i$ for $i=1,2$.

$N^*_{\Theta_1/X}$
is generated by linearly independent sections $s'_j, t'_j$ on the $U_j$ which are defined as follows:
$s'_1 = x$ and $t'_1 = w^2 + xz'w$ on $U_1$, $s'_2= z$ and $t'_2 = x'(w')^2+zw'$ on $U_2$, 
and $s'_3= z$ and $t'_3 = y$ on $U_3$.
We have $s'_1 = x's'_2$ and $t'_1 = x't'_2$ on $U_1 \cap U_2$, 
and $s'_2 = s'_3$ and $t'_2 = t'_3$ on $U_2 \cap U_3$.
Therefore we have 
\[
N^*_{\Theta_1/X} \cong \mathcal{O}_{\Theta_1}(1,0) \oplus \mathcal{O}_{\Theta_1}(1,0). 
\]
$N^*_{\Theta_2/X}$
is generated by linearly independent sections $s''_j, t''_j$ on the $U_j$ which are defined as follows:
$s''_1 = x$ and $t''_1 = w + xz'$ on $U_1$, $s''_2=x'z$ and $t''_2 = x'w'+z$ on $U_2$, 
and $s''_3= x''yz-z^2$ and $t''_3 = y$ on $U_3$.
We have $s''_1 = s''_2$ and $t''_1 = t''_2$ on $U_1 \cap U_2$, 
and $s''_2 = x''s''_3$ and $t''_2 = x''t''_3$ on $U_2 \cap U_3$.
Therefore we have 
\[
N^*_{\Theta_2/X} \cong \mathcal{O}_{\Theta_2}(0,1) \oplus \mathcal{O}_{\Theta_2}(0,1). 
\]
Hence 
\[
\text{Ext}^1_X(F_{R,i},F_j) \cong \text{Hom}_{\Theta}
(\mathcal{O}_{\Theta}(0,0) \oplus \mathcal{O}_{\Theta}(0,0),F_j) \cong 0
\]
for all $i,j = 1,2$.
By the duality, we conclude that 
$F_{R,1} \oplus F_{R,2}$ is a versal $2$-pointed NC deformation of $F_1 \oplus F_2$.

\vskip 1pc

The deformation algebra $R$ has the following form
\[
\left( \begin{matrix} 
\Bbbk+\Bbbk t^2 & \Bbbk t \\
\Bbbk t & \Bbbk+\Bbbk t^2
\end{matrix} \right) \mod t^3.
\]
$F_R$ is a relative $3$-spherical object over $R$:
\[
R\text{Hom}(F_R,F_i) \cong R/M_i \oplus R/M_i[-3]. 
\]
\end{Expl}

\begin{Expl}
Let $Y \subset \Bbbk^4$ be a hypersurface of dimension $3$ defined by an equation
$x_1x_2+x_3^2+x_4^3=0$.
The blowing up at the origin gives a resolution of singularities
$f: X \to Y$ with an exceptional divisor $E$, 
which is a quadric cone over $\mathbf{P}^1$ that was considered in Example~\ref{quadric2}.
We use the notation $\mathcal{O}_E(a)$ defined there.
We have $K_X = f^*K_Y+E$, $\mathcal{O}_E(E) = \mathcal{O}_E(-2)$ and 
$K_E = \mathcal{O}_E(-4)$.

Let $e = \mathcal{O}_E(-2)$.
Then $e$ is an exceptional object in $D^b(\text{coh}(X))$.
Let $\mathcal{D}$ be its left orthogonal complement, and let $S$ be the Serre functor of $\mathcal{D}$.
Then $F = \mathcal{O}_E(-1)$ is an object in $\mathcal{D}$.
If $S'$ is the Serre functor of $D^b(\text{coh}(X))$, then we have 
$S'(F) \cong \mathcal{O}_E(-3)[3]$.
Let $j_*: \mathcal{D} \to D^b(\text{coh}(X))$ be the inclusion functor, and 
$j^!: D^b(\text{coh}(X)) \to \mathcal{D}$ its right adjoint functor.
Then we have $S \cong j^!S'j_*$.
Since $j^!e \cong 0$, we deduce that $S(F) \cong j^!\mathcal{O}_E(-3)[3] \cong \mathcal{O}_E(-1)[2]$. 
From an exact sequence
\[
0 \to \mathcal{O}_E(-3) \to e^{\oplus 2} \to \mathcal{O}_E(-1) \to 0
\]
we deduce that $S(F) \cong F[2]$.

We construct a non-trivial self extension $G$ of $F$ as in Example~\ref{quadric2}.
$G$ is a versal NC deformation of $F$ over $R = \Bbbk[t]/(t^2)$, and 
$G$ is a relative $2$-spherical object in $\mathcal{D}$:
\[
R\text{Hom}(G,F) \cong R/M \oplus R/M[-2]. 
\]
Indeed we will show that $\text{Ext}^1(G,F) \cong 0$ in the following.
Then we have $\text{Ext}^2(G,F) \cong \Bbbk$ by the duality.
Let $i: E \to X$ be the embedding.
Then we have $i^*i_*G \cong G \oplus G(-E)[1]$, 
where the direct sum decomposition to cohomologies is consequence of the fact 
that $\text{Ext}^2(G,G(-E)) \cong \text{Hom}(G(-E),G)^* \cong 0$ since $-E$ is ample.
Hence 
\[
\text{Ext}^1_X(i_*G,i_*F) \cong \text{Ext}^1_E(G,F) \oplus \text{Hom}_E(G(-E),F) \cong 0.
\] 
\end{Expl}

\begin{Rem}
The category $\mathcal{D}$ in the above example was already considered in \cite{Kawamata}~4.3.
The sheaf $G$ there appeared in \cite{Toda-contraction}~4.13.
The construction of tilting generators in \cite{VdBergh} can also be considered as a multi-pointed 
non-commutative deformation of a collection which is not simple.
See also \cite{Toda-Uehara}.
\end{Rem}

We have a similar example in dimension $4$, where we obtain again a relative $2$-spherical object.
A non-trivial permutation $\sigma$ of the indexes appears in this example:

\begin{Expl}\label{singular 4fold}
Let $Y \subset \Bbbk^5$ be a hypersurface defined by an equation
$x_1x_2+x_3x_4+x_5^3=0$.
The blowing up at the origin gives a resolution of singularities
$f: X \to Y$ with an exceptional divisor $E$, which is a cone over $\mathbf{P}^1 \times \mathbf{P}^1$
that was considered in Example~\ref{quadric3}.
We use the notation $\mathcal{O}_E(a,b)$ defined there.
We have $K_X = f^*K_Y+2E$, $\mathcal{O}_E(E) = \mathcal{O}_E(-1,-1)$ and 
$K_E = \mathcal{O}_E(-3,-3)$.

Let $e_1 = \mathcal{O}_E(-1,-1)$ and $e_2 = \mathcal{O}_E(-2,-2)$.
Then $(e_2,e_1)$ is an exceptional collection in $D^b(\text{coh}(X))$.
Let $\mathcal{D}$ be its left orthogonal complement, and let $S$ be the Serre functor of $\mathcal{D}$.
Then $F_1 = \mathcal{O}_E(-1,0)$ and $F_2 = \mathcal{O}_E(0,-1)$ are objects in $\mathcal{D}$.
If $S'$ is the Serre functor of $D^b(\text{coh}(X))$, then we have 
$S'(F_1) \cong \mathcal{O}_E(-3,-2)[4]$.
From exact sequences
\[
\begin{split}
&0 \to \mathcal{O}_E(-3,-2) \to e_2^{\oplus 2} \to \mathcal{O}_E(-1,-2) \to 0 \\ 
&0 \to \mathcal{O}_E(-1,-2) \to e_1^{\oplus 2} \to \mathcal{O}_E(-1,0) \to 0
\end{split}
\]
we deduce that $S(F_1) \cong F_1[2]$.
Similarly we have $S(F_2) \cong F_2[2]$.

We construct non-trivial self extensions $G_1$ and $G_2$ of $F_1$ and $F_2$ 
as in Example~\ref{quadric3}, respectively.
Then $G = G_1 \oplus G_2$ is a versal $2$-pointed NC deformation of $F = F_1 \oplus F_2$ 
over $R = \left( \begin{matrix} \Bbbk & \Bbbk t \\ \Bbbk t & \Bbbk \end{matrix} \right) \mod t^2$.
By the vanishing theorem, we have $H^p(E,\mathcal{O}_E(a,b)) = 0$ for $p > 0$ if $a,b \ge -2$, and
$G$ is a relative $2$-spherical object in $\mathcal{D}$:
\[
R\text{Hom}(G,F_i) \cong R/M_i \oplus R/M_{3-i}[-2]. 
\]
Indeed we will show that $\text{Ext}^1(G_j,F_k) \cong 0$ for $j,k = 1,2$ in the following.
Then we have $\text{Ext}^2(G_i,F_{3-i}) \cong \Bbbk$ by the duality.
Let $i: E \to X$ be the embedding.
Then we have $i^*i_*G_j \cong G_j \oplus G_j(-E)[1]$.
Hence 
\[
\text{Ext}^1_X(i_*G_j,i_*F_k) \cong \text{Ext}^1_E(G_j,F_k) \oplus \text{Hom}_E(G_j(-E),F_k) \cong 0. 
\]
\end{Expl}



Graduate School of Mathematical Sciences, University of Tokyo,
Komaba, Meguro, Tokyo, 153-8914, Japan 

kawamata@ms.u-tokyo.ac.jp


\begin{thebibliography}{}

\bibitem{Ann-Log}
Anno, Rina; Logvinenko, Timothy. 
{\em Spherical DG-functors}. 
arXiv:1309.5035 

\bibitem{AR}
Auslander, Maurice; Reiten, Idun.
{\em Cohen-Macaulay and Gorenstein Artin rings}.
Progress in Math. 95(1991), 221--245.

\bibitem{BB}
Bodzenta, Agnieszka; Bondal, Alexey.
{\em Flops and spherical functors}.
arXiv:1511.00665

\bibitem{BK}
Bondal, Alexei; Kapranov, Mikhail.
{\em Representable functors, Serre functors, and reconstructions}.
Izv. Akad. Nauk SSSR Ser. Mat., 6-53 (1989),
1183--1205.

\bibitem{Bridgeland}
Bridgeland, Tom.
{\em Equivalences of triangulated categories and Fourier-Mukai transforms}. 
Bull. London Math. Soc. 31 (1999), no. 1, 25--34.

\bibitem{Donovan-Wemyss1}
Donovan, Will; Wemyss, Michael.
{\em Noncommutative deformations and flops}. 
arXiv:1309.0698 
	
\bibitem{Donovan-Wemyss2}
Donovan, Will; Wemyss, Michael.
{\em Twists and braids for general 3-fold flops}. 
arXiv:1504.05320

\bibitem{KMM}
Kawamata, Yujiro; Matsuda, Katsumi; Matsuki, Kenji. 
{\em Introduction to the minimal model problem}. 
in Algebraic Geometry Sendai 1985,
Advanced Studies in Pure Math. {\bf 10} (1987), 
Kinokuniya and North-Holland, 283--360. 

\bibitem{stacks}
Kawamata, Yujiro.
{\em Equivalences of derived categories of sheaves on smooth stacks}.
Amer. J. Math. {\bf 126}(2004), 1057--1083.

\bibitem{Kawamata}
Kawamata, Yujiro.
{\em Derived categories and birational geometry}. 
Algebraic geometry—Seattle 2005. Part 2, 655--665, Proc. Sympos. Pure Math., 80, Part 2, 
Amer. Math. Soc., Providence, RI, 2009.

\bibitem{Kuznetsov}
Kuznetsov, Alexander.
{\em Derived categories of quadric fibrations and intersections of quadrics}.
Advances in Mathematics 218(5), (2008), 1340--1369.

\bibitem{Laudal}
Laudal, O. A. 
{\em Noncommutative deformations of modules}. 
The Roos Festschrift volume, 2. Homology Homotopy Appl. 4 (2002), no. 2, part 2, 357--396.

\bibitem{Laudal2}
Laudal, O. A.
{\em Matric Massey products and formal moduli. I}.
Algebra, algebraic topology and their interactions (Stockholm, 1983), 218--240, 
Lecture Notes in Math., 1183, Springer, Berlin, 1986. 

\bibitem{Meachan}
Meachan, Ciaran. 
{\em A note on spherical functors}. 
arXiv:1606.09377 

\bibitem{Neeman}
Neeman, Amnon.
{\em The Grothendieck duality theorem via Bousfield's techniques and Brown representability}.
Journal of the American Mathematical Society, 9 (1996), 205--236.

\bibitem{Orlov}
Orlov, Dmitri.
{\em Projective bundles, monoidal transformations, and de- rived categories of coherent sheaves}. 
Russian Acad. Sci. Izv. Math. 41(1993), 133--141.

\bibitem{Schlessinger}
Schlessinger, Michael.
{\em Functors of Artin rings}. 
Trans. Amer. Math. Soc. 130 (1968), 208--222. 

\bibitem{Toda-twist}
Toda, Yukinobu.
{\em On a certain generalization of spherical twists}. 
Bull. Soc. Math. France 135 (2007), no. 1, 119--134. 

\bibitem{Toda-contraction}
Toda, Yukinobu.	
{\em Stability conditions and extremal contractions}.
Math. Ann. 357 (2013), no. 2, 631--685. 
	
\bibitem{Toda-GV}
Toda, Yukinobu.
{\em Non-commutative width and Gopakumar-Vafa invariants}.
Manuscripta Math. 148 (2015), no. 3-4, 521--533.

\bibitem{Toda-SLC}
Toda, Yukinobu.
{\em Non-commutative thickening of moduli spaces of stable sheaves}.
a talk at AMS Summer Institute in Algebraic Geometry, Salt Lake City 2015, 
video recording at http://www.claymath.org/utah-videos

\bibitem{Toda-Uehara}
Toda, Yukinobu; Uehara, Hokuto.
{\em Tilting generators via ample line bundles}. 
Adv. Math. 223 (2010), no. 1, 1--29.

\bibitem{VdBergh}
Van den Bergh, Michel.
{\em Three-dimensional flops and noncommutative rings}. 
Duke Math. J. 122 (2004), no. 3, 423--455. 

\end{thebibliography}
\end{document}